\documentclass[reqno, 11pt, a4paper]{amsart} 
\usepackage[english]{babel}
\usepackage{amsfonts, amsmath, amsthm, amssymb,amscd,indentfirst}
\usepackage{mathtools}

\usepackage[dvipsnames]{xcolor}
\usepackage{newtxtext,newtxmath}
\usepackage{enumerate}
\usepackage{enumitem}
\usepackage{esint}
\usepackage{stackrel}
\usepackage{anysize} 
 \usepackage[mathscr]{euscript}  
 
 \usepackage[left=2.5cm, right=2.5cm, top=2cm, bottom=2cm]{geometry}

\usepackage{microtype}    

\usepackage{etoolbox}  

\makeatletter  

\patchcmd{\@tocline}
  {\hfil}  
  {\leaders\hbox to 0.6em{\hss.\hss}\hfill}  
  {}{}  

\renewcommand{\@pnumwidth}{1.5em}  

\def\l@section      {\@tocline{1}{0.6em}{0em}{}{}}  

\def\l@subsection   {\@tocline{2}{0.3em}{2em}{}{}}  

\def\l@subsubsection{\@tocline{3}{0.2em}{4.7em}{}{}}  

\patchcmd{\@tocline}
  {\vskip #1}  
  {\vskip #1\relax}  
  {}{}

\makeatother  

\theoremstyle{definition}
\newtheorem{example}{Example}[section]
\newtheorem{theorem}{Theorem}[section]
\newtheorem{proposition}[theorem]{Proposition}
\newtheorem{lemma}[theorem]{Lemma}
\newtheorem{definition}[theorem]{Definition}
\newtheorem{corollary}[theorem]{Corollary}
\newtheorem{remark}[theorem]{Remark}

 \numberwithin{equation}{section}

\newcommand{\revc}{{\bf r}}   
\newcommand{\grad}{\boldsymbol{\nabla}}  
\newcommand{\gaglia}{\gamma_{p^-}}  
\newcommand{\deriv}{\mathcal{D}}   
\newcommand{\vltr}{\mathcal{V}}
\newcommand{\consgaglia}{\mathtt{C}_{p^-}}  
 \newcommand{\ds}{\textnormal{d}A_{\vartheta}}
 \newcommand{\spt}{\operatorname{supp}}

\newcommand{\wghtv}{\vartheta}
\newcommand{\loc}{\operatorname{loc}}

\newcommand{\fnf}{F}
\newcommand{\fnfs}{F ^\circ}
\newcommand{\rmf}{\mathcal M}
\newcommand{\ctb}{\mathscr{C}}
\newcommand{\rdo}{\mathcal{R} _0}

\newcommand{\esssup}{\operatorname{ess}\sup}
\newcommand{\essinf}{\operatorname{ess}\inf}

\newcommand{\piu}{\rho_{1}}
\newcommand{\pid}{\rho_{2}}
\newcommand{\din}{\textnormal{d}}

\title{Pointwise Estimates Near Singular Sets for Quasilinear Elliptic Equations}

\author{Juan Pablo Alcon Apaza}

\address{Departamento de Matem\'atica, Universidade Federal de São Carlos, 13565-905, São Carlos--SP, Brazil}

\email{juanpabloalconapaza@gmail.com}

\begin{document}

\begin{abstract}
In this work, we study the removability of boundary singular sets for certain
classes of quasilinear elliptic equations in domains $\Omega$ of an
$n$-dimensional Finsler manifold $(\rmf,\fnf,\wghtv)$. We work with
Lipschitz functions $\piu$ and $\pid$ satisfying distance-type properties; in
particular, $\fnf(\cdot,\grad \piu)\le 1$ and $\fnf(\cdot,\grad \pid)\le 1$ a.e. in $\rmf$. The singular set is defined by $\Gamma=\piu^{-1}(\{0\})$. The model problem is
$-\Delta_{p(x)}u+|u|^{q-1}u=0$
in domains of $\mathbb R^n\cong
\mathbb R^d\times\mathbb R^{n-d}
\cong
\piu^{-1}(\{0\})\times\pid^{-1}(\{0\})$,
where $
\piu(x)=|(x_{d+1},\ldots,x_n)|$
 and $\pid(x)=|(x_1,\ldots,x_d)|$.

The main tool in our analysis is the estimate
\[
|u(x)|\leq \mathbf C\,\piu(x)^{-\tau}
\]
near $\Gamma$ for weak solutions
$u\in
W^{1,p(x)}_{\operatorname{loc}}
\bigl(\overline{\Omega}\setminus(\Gamma\cup\Sigma);\wghtv\bigr)
\cap
L^\infty_{\operatorname{loc}}
\bigl(\overline{\Omega}\setminus(\Gamma\cup\Sigma)\bigr)$,
where the constants $\mathbf C>0$ and $\tau>0$ converge to positive values as
$p^+\to1$. This estimate is a key ingredient in proving that the singularity
at $\Gamma$ is removable.

Moreover, in a bounded domain $\Omega$, using this estimate and assuming that,
for every variable exponent satisfying $1<p^-\le p^+<\min\{2,q+1\}$, there exists a weak solution
$u_p\in
W^{1,p(x)}_{\operatorname{loc}}(\Omega;\wghtv)
\cap
L^\infty_{\operatorname{loc}}(\Omega)$
of
\[
-\operatorname{div}
\left(
|\grad u_p|_{\fnf}^{p-2}\grad u_p
\right)
+
|u_p|^{q-1}u_p=0
\quad\text{in }\Omega,
\]
we prove that, for every $\mathscr U\Subset\Omega$, there exists a subsequence
$\{u_{p_m}\}$, with $p_m^+\to1$, that converges to a solution
$
u\in BV(\mathscr U;\wghtv)\cap L^{q+1}(\mathscr U;\wghtv)
$
of
$$
-\Delta_1u+|u|^{q-1}u=0
\quad\text{in }\mathscr U.
$$

\end{abstract}

\let\thefootnote\relax
\let\thefootnote\relax





\maketitle

\tableofcontents
\section{Introduction and main results}
This work is concerned with the problem of removable singularities $\Gamma$ for weak solutions of nonlinear elliptic equations with boundary conditions of the form
\begin{equation} \label{1}
\left\{
\begin{aligned}
-\operatorname{div} \left[  \mathscr{A}_p (\cdot ,\grad u) \right]  + f(\cdot,u)& = 0 & & \text{in } \Omega,\\
g_{\nu}\left(\nu ,  \mathscr{A}_p (\cdot ,\grad u) \right) +  b(\cdot , u)   &= 0 & & \text{on }  \partial \Omega \setminus \Sigma,
\end{aligned}
\right.
\end{equation}
where \(g_{\nu}\bigl(\nu,\mathscr{A}_p(\cdot,\nabla u)\bigr)\)
denotes the conormal derivative associated with the outward unit normal
\(\nu\) to \(\partial\Omega\).

Throughout this work, $\Omega$ denotes a domain in an $n$-dimensional Finsler manifold
$(\rmf,\fnf,\wghtv)$, where $n \ge 2$, while $\Sigma$ is a closed subset of
$\partial\Omega$ satisfying $\Gamma \subset \partial\Omega \setminus {\Sigma}$.

We assume the existence of Lipschitz functions
\[
\piu : \rmf \to [0,\infty), \qquad \pid : \rmf \to [0,\infty),
\]
and of a constant $\rdo \in (0,\tfrac12)$ satisfying
\begin{equation}\label{278}
\left\{\begin{gathered}
 \fnf(\cdot,\grad \rho_i) \le 1 
 \quad \text{a.e. in } \rmf \quad \text{for } i=1,2, \\
 \piu^{-1}(\{t\}) \neq \emptyset
 \quad \forall\, t \in [0,2\rdo],\\
 \piu^{-1}([0,r_1)) \cap \pid^{-1}([0,r_2))
 \neq \emptyset \text{ and is bounded}
 \quad \forall\, (r_1,r_2)\in (0,2\rdo]\times \mathbb{R}^+,\\
 \piu^{-1}(\{0\})=\Gamma, 
 \qquad \sup_{x\in \Omega} \piu(x) <\infty.
\end{gathered} \right.
\end{equation}
Moreover, $\piu$ and $\pid$ satisfy conditions \ref{102}--\ref{104} on page~\pageref{102}.

We also assume that the {\it reversibility constant} of $(\rmf , \fnf)$ is finite; that is,
\begin{equation}\label{276}
\revc_{\fnf} :=\sup _{x \in \rmf} \  \sup _{Y \in T_{x} \rmf  \setminus \{0\}} \frac{\fnf (x, Y)}{\fnf(x,-Y)} <\infty .
\end{equation}

The following examples show that the above assumptions cover both
unbounded cylindrical domains and bounded domains.
\begin{example}$ $
\begin{enumerate}
 \item The model example under consideration is
\begin{gather*}
\Omega = \mathbb{R}^d \times B^{n-d,+}_1(0), 
\qquad 
\Gamma = \mathbb{R}^d \times \{0\}, 
\qquad 
\Sigma = \mathbb{R}^d \times 
\bigl(\partial B^{n-d}_1(0)\cap \{x_n\geq 0\}\bigr),\\
\piu(x)= |x'|_{\mathbb{R}^{n-d}}, 
\qquad 
\pid(x)= |x''|_{\mathbb{R}^{d}},
\end{gather*}
where $x'=(x_{d+1},x_{d+2},\ldots,x_n)$, 
$x''=(x_1,x_2,\ldots,x_d)$, and $|\cdot|_{\mathbb{R}^i}$ denotes the Euclidean norm on $\mathbb{R}^i$. Here
\[
B^{n-d,+}_1(0)
=
\left\{
y=(x_{d+1},\ldots,x_n)\in \mathbb{R}^{n-d}
\mid 
|y|_{\mathbb{R}^{n-d}}<1,\ x_n>0
\right\}.
\]

\item Another example is
\begin{gather*}
\Omega = \mathbb{R}^d \times B^{n-d}_1(e_n), 
\qquad 
\Gamma = \mathbb{R}^d \times \{0\},
\qquad  
\Sigma =\mathbb{R}^d \times 
\bigl(\partial B^{n-d}_1(e_n)\cap \{x_n\geq 1/2\}\bigr),\\
\piu(x)= |x'|_{\mathbb{R}^{n-d}}, 
\qquad 
\pid(x)= |x''|_{\mathbb{R}^{d}} .
\end{gather*}

\item Finally, one may also consider
\begin{gather*}
\Omega = B^{n,+}_1(0) 
\quad \text{or} \quad 
\Omega = B^n_1(0),\\
\pid \equiv 0, 
\qquad 
d=0, 
\qquad 
\piu(x)= |x|_{\mathbb{R}^{n}}.
\end{gather*}
 \end{enumerate}
 \end{example}

The variable exponent \(p : \Omega \to (1,\infty)\) is measurable, and \(\Omega\), \(p\), and \(q\) are assumed to satisfy the conditions stated in \eqref{224}--\eqref{243} and \ref{102}--\ref{106} below. Moreover, the operator $\mathscr{A}_p: T\Omega \to T\Omega$ is assumed to satisfy the following conditions for every function $u: \Omega \to \mathbb{R}$ that is differentiable in the sense of distributions:
\begin{enumerate}[label=($a_{\arabic*}$)]
\item \label{fnf16} $\left[ \mathscr{A}_p(\cdot , \grad u  )  \bullet  \deriv u \right] (x) \geq \Lambda _1 {\big[}  | \grad u |_{\fnf}  (x) {\big]}^{p(x)} $ for a.e. $x \in \Omega$.

\item \label{fnf21} $  |\mathscr{A}_p( \cdot , \grad u ) |_{\fnf} (x )\leq \Lambda _2  {\big[} | \grad u |  _{\fnf}  (x)  {\big]} ^{p(x)-1}$ for a.e. $x \in \Omega$.

\end{enumerate}
Here, we write
$$
\fnf (x , Y  ) = |Y|_{\fnf} (x) \quad \text { and }  \quad \fnfs (x , \omega  ) = |\omega|_{\fnfs} (x),
$$
for $Y\in T_{x} \Omega$ and $\omega \in T^\ast  _{x}\Omega$.

Furthermore, for a function \( f: \Omega \rightarrow \mathbb{R} \) that is differentiable in the distributional sense, we use the notation
$$
Y\bullet \deriv f (x)= \deriv f _{x}(Y) \quad \text { for } \quad Y\in T_{x} \Omega,
$$
where \( \deriv f_{x} \in T_{x}^\ast \Omega \) (or equivalently \( \deriv f(x) \)) denotes the derivative of \( f \) at \( x \).

If we define
\[
\tilde{\mathscr{A}}_p(x,Y) := -\mathscr{A}_p(x,Y),
\]
then, by \eqref{276}, \ref{fnf16}, and \ref{fnf21}, we obtain
\begin{enumerate}[label=($\tilde{a}_{\arabic*}$)]
\item   \label{135} $\left[ \tilde{\mathscr{A}}_p(\cdot , \grad (-v)  )  \bullet  \deriv v \right] (x) \geq \revc_{\fnf}^{-p^+}  \Lambda _1 \big[  | \grad v |_{\fnf}  (x) \big]^{p(x)} $ for a.e. $x \in \Omega$.

\item  \label{136} $  |\tilde{\mathscr{A}}_p( \cdot , \grad (-v) ) |_{\fnf} (x )\leq   \revc_{\fnf}^{p^+}  \Lambda _2 \big[ | \grad v |  _{\fnf}  (x)  {\big]} ^{p(x)-1}$ for a.e. $x \in \Omega$.

\end{enumerate}
Here we use the notation
$$
p^-=\underset{\Omega}{\essinf} \, p, \qquad p^+=\underset{\Omega}{\esssup} \, p.
$$
And we suppose
\begin{equation}\label{224}
1 < p^- \leq  p^+ < n_d \le n,\qquad 0 \le d < n,
\end{equation}
for some $n_d, d \in \mathbb{N}$.

Regarding \(f\) and \(b\), we assume that $f : \Omega \times \mathbb{R} \to \mathbb{R}$ and $b : \partial \Omega \times \mathbb{R} \to \mathbb{R}$ are measurable functions. Moreover, both \(f\) and \(b\) are locally bounded and satisfy the following condition:
\begin{enumerate}[label=($H_{\arabic*}$)]
\item \label{43} $f(x,u)\operatorname{sign}u \ge |u|^{q}$ for a.e. $x\in\Omega$ and all $u\in\mathbb R$, and
$b(x,u)\operatorname{sign}u \ge 0$ for a.e. $x\in\partial\Omega$ and all $u\in\mathbb R$. Here,
\begin{gather}
q - p^+ + 1 > 0, \label{225} \\
n_d > \underset{x\in\Omega}{\operatorname{ess\,sup}}\, \frac{p(x)q}{q-p(x)+1}. \label{243}
\end{gather}
\end{enumerate}

\subsection{Context and Related Work} 
The study of singularities of solutions to partial differential equations is a fundamental topic in elliptic theory and nonlinear potential analysis. In general terms, the removability problem consists of determining whether a solution defined outside a singular set can be extended across that set as a solution of the same equation. Singularities of solutions to second-order quasilinear equations were extensively studied by V\'eron in his monograph \cite{zbMATH00939333}. For results on the removability of singularities of solutions to elliptic equations with absorption terms, see \cite{zbMATH01675810, zbMATH04096004, liskevich2008isolated, skrypnik2005removability}. For the removability of singularities in second-order equations in the anisotropic case, see \cite{zbMATH05135061, zbMATH06673573}. For related problems involving unbounded singular sets, we refer to \cite{Bidaut-Veron2014localand, caffarelli2022singular, fakhi2000new, juutineremovabilitylevel2005}. Moreover, the topic is closely related to the study of the asymptotic behavior of solutions to nonlinear elliptic equations; see, for example, \cite{friedman1986singular, vazquez1985isolated}.


Removability is not determined solely by the size of the singular set, but also by the way in which the solution may blow up near it. In this sense, pointwise estimates of the form
\[
    |u(x)| \lesssim |x-x_0|^{-\tau},
    \qquad \tau >0,
\]
play an important role; see, for example, \cite{zbMATH04096004, friedman1986singular, vazquez1985isolated}. Another condition on the behavior of solutions near a singularity that implies removability can be found in \cite{nicolosi2003precise}:
\[
\limsup_{r\to 0} M(r)\, r^{(n-p)/(p-1)}=0,
\]
where $1<p<n$,
\[
M(r):=\max \left\{|u(x)| \mid r\leq |x-x_0|\leq R_0\right\},
\]
and $0<R_0<\min \{\operatorname{dist}(\{x_0\},\partial\Omega),1\}$.


Let us recall some well-known results concerning isolated singular points for solutions of
\begin{equation}\label{eq:quasilinear}
    -\sum_{j=1}^{n} \bigl( a_j(x,\nabla u) \bigr)_{x_j}
    + \tilde{g}(x,u)
    = 0
    \qquad \text{in } \Omega,
\end{equation}
where $\Omega$ is a bounded domain in $\mathbb{R}^{n}$, the functions $a_j(x,\eta)$, $j=1,\ldots,n$, satisfy the standard ellipticity and growth conditions with respect to $x$ and $\eta$, and $\tilde g$ satisfies suitable structural assumptions. Serrin \cite{serrin1964localbehavior, serrin1965isolated} showed that
\[
    u(x)
    =
    O\!\left(
        |x-x_{0}|^{-\frac{n-p}{p-1}+\delta}
    \right),
    \qquad
    \delta >0,
    \qquad
    1<p<n,
\]
is a sufficient condition for the removability of the singularity at $x_0$. A more precise condition for the removability of isolated singularities can be found in \cite{nicolosi2003precise}; namely,
\[
    \max_{r\leq |x-x_0|\leq R_0} |u|
    =
    o\!\left(
        r^{-\frac{n-p}{p-1}}
    \right),
    \qquad
    1<p<n.
\]
A generalization to the anisotropic case can be found in \cite{zbMATH05135061}.


The removability of singularities of solutions of
\[
    -\Delta u + \tilde g(x,u) = 0
    \qquad \text{in } \Omega
\]
was studied by Br\'ezis and V\'eron \cite{brezis1980removable}. They proved the removability of the isolated singularity $x_0=0$ for problems satisfying the growth condition
\[
    \tilde g(x,u)\operatorname{sign}u \geq |u|^{q},
    \qquad
    q \geq \frac{n}{n-2},
    \qquad
    n \geq 3.
\]
Further developments in this direction can be found in \cite{cianci2010removability}.


For \eqref{eq:quasilinear}, Skrypnik \cite{skrypnik2005removability} proved the removability of the singular point for equations satisfying the conditions
\[
    \tilde g(x,u)\operatorname{sign}u \geq |u|^{q},
    \qquad
    1<p<n,
    \qquad
    q \geq \frac{n(p-1)}{n-p}.
\]
Moreover, he showed the following behavior near the singular set:
\[
|u(x)| \lesssim |x|^{-\frac{p}{q-p+1}}.
\]

Liskevich and Skrypnik \cite{liskevich2008isolated} established the removability of the isolated singular point $x_0=0$ for \eqref{eq:quasilinear} when
\[
    \tilde g(x,u) \operatorname{sign} u \approx |x|^{-\alpha} |u|^q,
\]
provided that
\[
    q \geq \frac{(p-1)(n-\alpha)}{n-p},
    \qquad
    \alpha < p,
    \qquad
    1<p<n.
\]
Moreover, they provided the following estimate:
\[
|u(x)| \lesssim |x|^{-\frac{p-\alpha}{q-p+1}}.
\]

Weighted and degenerate variants of this problem have also been studied. When
\[
    \sum_{j=1}^{n} a_j(x,\eta)\eta_j
    \gtrsim
     w_1(x)|\eta|^{p},
    \qquad
    |a_j(x,\eta)|
    \lesssim
    w_1(x)|\eta|^{p-1},
\]
and
\[
    \tilde g(x,u)\operatorname{sign}u
    \geq
    w_0(x)|u|^{q},
\]
Mamedov and Harman \cite{mamedov2009removability} proved that an isolated singular point $x_0$ is removable for solutions of \eqref{eq:quasilinear} if
\[
    w_0(B(x_0,\epsilon))
    \left(
        \frac{
            w_1(B(x_0,\epsilon))
        }{
            \epsilon^{p} w_0(B(x_0,\epsilon))
        }
    \right)^{\frac{q}{q-p+1}}
    =
    o(1)
    \qquad \text{as } \epsilon \to 0,
\]
with $p>1$ and $q>p-1$.

Another important direction concerns problems with variable exponents. In the framework of Sobolev spaces with variable exponents, Fu and Shan \cite{fu2016removability} obtained the following behavior near the singular point $x_0$:
\[
\sup_{x\in B(x_0,r/2)} |u(x)| \lesssim r^{-\tau},
\qquad
0<r<\delta,
\]
where
\[
\tau > \frac{p^+_{B(x_0,\delta)}}{q_{B(x_0,\delta)} - p^+_{B(x_0,\delta)} +1}.
\]
Moreover, they obtained sufficient conditions for the removability of isolated singular points for elliptic equations with $p(x)$-type growth. More precisely, if
\[
    \sum_{j=1}^{n} a_j(x,\eta)\eta_j
    \gtrsim
     |\eta|^{p(x)},
    \qquad
    |a_j(x,\eta)|
    \lesssim
      |\eta|^{p(x)-1},
\]
and
\[
    \tilde g(x,u)\operatorname{sign}u
    \geq
    |u|^{q},
\]
where $p,q:\overline{\Omega}\to\mathbb{R}$ are continuous functions, then an isolated point $x_0$ is removable provided that
\[
    1 < \inf_{\overline{\Omega}} p
    \leq
    \sup_{\overline{\Omega}} p
    < n,
    \qquad
    \inf_{\overline{\Omega}}(q-p+1) > 0,
\]
and
\[
    \sup_{\overline{\Omega}}
    \frac{pq}{q-p+1}
    < n.
\]

A different approach, for unbounded singular sets, is given by equations with absorption or source terms and measure data:
\begin{equation}\label{eq:measure-data}
    -\Delta_p u + |u|^{q-1}u = \mu,
\end{equation}
in a domain $\Omega \subset\mathbb{R}^{n}$, where $1<p<n$, $q>p-1$, and $\mu$ is a Radon measure in $\Omega$. In \cite{Bidaut-Veron2003removsinquasi}, the removability result is expressed in terms of the Bessel capacity condition
\[
    \operatorname{cap}_{1,s}( \mathcal{S},\mathbb{R}^{n}) = 0,
    \qquad
    s > \frac{pq}{q-p+1},
\]
where $\mathcal{S}$ is a relatively closed subset of $\Omega$. In this context, the problem consists of determining whether a local entropy solution of \eqref{eq:measure-data} in $\Omega\setminus\mathcal S$ extends as a local entropy solution to the whole domain. In the semilinear case $p=2$, Baras and Pierre \cite{barras1984singu} showed that the removable sets are precisely those sets $\mathcal{S}$ satisfying
\[
    \operatorname{cap}_{2,q/(q-1)}(\mathcal{S},\mathbb{R}^{n}) = 0.
\]
For nonlinear boundary conditions,
\[
\left\{
\begin{aligned}
    -\Delta_p u &= 0
    &&\text{in } \mathbb{R}^{n}_{+},\\
    |\nabla u|^{p-2}\frac{\partial u}{\partial \nu}
    &= u^{q}
    &&\text{on } \partial\mathbb{R}^{n}_{+},
\end{aligned}
\right.
\]
Aguirre \cite{aguirre2019harmonic} characterized removable compact sets $K\subset\partial\mathbb{R}^{n}_{+}$ by means of Riesz capacities, showing that they are precisely the sets for which
\[
    \operatorname{cap}_{I_{p-1},\,q/(q-p+1),\,n-1}(K) = 0,
\]
when
\[
    1<p<n,
    \qquad
    q > \frac{(n-1)(p-1)}{n-p}.
\]
\medskip


\noindent {\it The limit problem as $p\to 1$.} The limiting problems associated with the $p$-Laplacian have been investigated by several authors; see, for instance, \cite{4AndreuBall2001, 6Andreaselles2004, 9Bellvaga2005, 11Demengel2002, 12Demengel2004, zbMATH02228525, zbMATH05351718, zbMATH07786723} and the references therein. A further source of motivation for this problem comes from the variational approach to image restoration introduced by Rudin, Osher, and Fatemi. Interest in the study of the behaviour of $u_p$ also arises in optimal design problems in torsion theory and in related geometric questions; see \cite{16Kawohl1990, 17Kawohl1991}. 

A natural approach to the analysis of solutions to
\begin{equation}\label{ptou2}
-\operatorname{div}\left(\frac{Du}{|Du|}\right)=f
\quad \text{in } \Omega,
\end{equation}
subject to homogeneous Dirichlet boundary conditions, is to study the asymptotic behaviour, as $p\to1$, of the solutions $u_p$ to
\[
\left\{\begin{aligned}
-\operatorname{div}\left(\left|\nabla u_p\right|^{p-2}\nabla u_p\right)=&f & & \text{in } \Omega,\\[4pt]
u_p=&0 & & \text{on } \partial\Omega,
\end{aligned}\right.
\]
where $p>1$ and $\Omega$ is a bounded open subset of $\mathbb{R}^N$, with $N\ge 2$, and Lipschitz boundary. See \cite{zbMATH05346734, 10Cicabetti2003} for the case in which the datum $f$ belongs to the Lorentz space $L^{N,\infty}(\Omega)$. Solutions to \eqref{ptou2} belong to the space $BV(\Omega)$, which ensures that the distributional gradient is well defined as a Radon measure.

The dependence of the eigenvalues of $-\Delta_p$ on $p$ has also attracted considerable attention in the literature. This line of investigation goes back at least to \cite{17delPinoManasevich1991}, where the continuity with respect to $p$ of the first Dirichlet eigenvalue $\lambda_1^{\mathcal D}$ was established. As for higher Dirichlet eigenvalues, the continuity of $\lambda_2^{\mathcal D}$ is stated in \cite{25Huang1997}, while the continuity of the full sequence $\lambda_n^{\mathcal D}$ with respect to $p$ is discussed in detail in \cite{11ChampionDePascale2007}; see also \cite{14DegiovanniMarzocchi2014, 40Parini2011}.

The analysis of the limit $\lambda_1^{\mathcal D}\to\bar{\lambda}_1$ of the first Dirichlet eigenvalue as $p\to1$, and of the role of $\bar{\lambda}_1$ as the first eigenvalue of the $1$-Laplacian operator $-\Delta_1$, can be found in \cite{19Demengel2002, 12Demengel2004, 28KawohlFridman2003, 29Kraiem2007}, with preliminary approximation results in \cite{25Huang1997, 31LeftonWei1997}. The problem of higher Dirichlet eigenvalues is addressed in \cite{34LittigSchuricht2014, 38Parini2009, 39Parini2010}. For the Robin problem, we refer to \cite{zbMATH07882872, zbMATH07751769}.

\subsection{Main results}  $ $

\noindent {\it Removable singularities.} We will follow the techniques developed in \cite{mamedov2009removability}, which were later used to obtain results in Sobolev spaces with variable exponent in \cite{fu2016removability}, as well as for unbounded singular sets in \cite{zbMATH07850623}. In this work, we explore these ideas in the context of domains in Finsler manifolds.

One of the main objectives is to prove the estimate, stated in Proposition \ref{266}, near the singular set $\Gamma = \piu^{-1}(\{0\})$:
$$
|u(x)| \leq {\bf C} \piu(x)^{-\tau}.
$$
In addition, we analyze the behavior of the constants ${\bf C}$ and $\tau$ as $p^+ \to 1$.

As we will see in Section \ref{32}, inequality \eqref{239}, the main tool for obtaining this estimate is a Gagliardo--Nirenberg--Sobolev type inequality:
\begin{equation*}
\left(\int_{\Omega} |u|^{\gaglia}\,\din \wghtv \right)^{\frac{1}{\gaglia}}
\le \consgaglia
\left(\int_{\Omega} |\grad u|_{\fnf}^{p^-}\, \din \wghtv \right)^{\frac{1}{p^-}} .
\end{equation*}
This inequality will be assumed as part of the hypotheses in \ref{106}. Since it applies to compactly supported functions, in the case where $\pid \not\equiv 0$, the decay condition imposed in \ref{102},
$$
u^{\pm}(x)\leq \ctb_0 m^{\pm}_{r_1,r_2} \left(1 + \pid(x) \right)^{-\boldsymbol{\delta}},
$$
will allow us to construct a compactly supported test function; see \eqref{300}.

Moreover, in Section \ref{32}, we will work mainly with subsolutions of the operator $\mathscr{A}_p$; see Lemma \ref{219}. An important point is that, in general, one does not have $\grad(-u) = -\grad u$. Therefore, we cannot assume that
$$
\mathscr{A}_p(x,\grad(-u)) = -\mathscr{A}_p(x,\grad u).
$$
This difficulty can be overcome by assuming that the reversibility constant of $(\rmf,\fnf)$ is finite. Under this hypothesis, we can work with the operator
$$
\widetilde{\mathscr{A}}_p(x,Y)= -\mathscr{A}_p(x,Y),
$$
in order to handle supersolutions; see \ref{135}, \ref{136}, and Corollary \ref{277}.

In Section \ref{158}, the estimate $|u(x)| \leq {\bf C} \piu(x)^{-\tau}$ will allow us to prove that
$$
(1+\pid)^{\boldsymbol{\delta}} u \in L^{\infty} _{\loc}(\overline \Omega \setminus \Sigma).
$$
Then, in Section \ref{33}, inequality \eqref{246}, this assertion will allow us to show that
$$
\int_{\Omega \cap \{\piu \leq 2r\}} |\grad u|^p_{\fnf}
(1+\pid)^{\boldsymbol{\delta} p^+ /2} \, \din \wghtv < \infty,
$$
and consequently that $|\grad u|_{\fnf} \in L^{p(x)} _{\loc}(\overline \Omega \setminus \Sigma ;\wghtv)$. With these tools, we will be able to remove the singularity $\Gamma$.

We work with the following spaces
\begin{multline*}
L_{\operatorname{loc} }^{p (x) }(\overline{\Omega} \setminus ( \Gamma \cup \Sigma ); \wghtv ) =  \Bigl\{ u : \Omega \rightarrow \mathbb{R} \mid u \in L^{p (x) }(U ; \wghtv  ) 
\text {  for all open subset }  U \subset \Omega \Bigr. \\
 \Bigl. \text { so that }  \overline{U} \cap ( \Gamma \cup \Sigma )=\emptyset \Bigr\}
\end{multline*}
and 
\begin{multline*}
W_{\operatorname{loc} }^{1, p(x)  }(\overline{\Omega} \setminus ( \Gamma \cup \Sigma ); \wghtv )=\Bigl\{u : \Omega \rightarrow \mathbb{R} \mid u \in W^{1, p(x)  }(U  ; \wghtv ) 
\text { for all open subset } U \subset \Omega  \Bigr.\\
 \Bigl. \text { so that }   \overline{U} \cap ( \Gamma \cup \Sigma )=\emptyset\Bigr\}.
\end{multline*}

Similarly, we define $L_{\operatorname{loc} }^{\infty}(\overline{\Omega} \setminus ( \Gamma \cup \Sigma ) )$. Let us observe that the trace of $u \in W_{\operatorname{loc} }^{1, p (x) }(\overline{\Omega} \setminus ( \Gamma \cup \Sigma ); \wghtv ) \cap L_{\operatorname{loc}  }^{\infty}(\overline{\Omega}   \setminus ( \Gamma \cup \Sigma ))$ may not be defined on $\partial \Omega  $; however, it is defined on $\{x \in \partial \Omega   \mid \piu (x)  >r\}$ and is essentially bounded for small enough values $r>0$.

\begin{definition}\label{298}
We say that $u \in W_{\operatorname{loc}}^{1,p(x)}(\overline{\Omega} \setminus ( \Gamma \cup \Sigma ); \wghtv)\cap L_{\operatorname{loc}}^{\infty}(\overline{\Omega}\setminus (\Gamma \cup \Sigma) )$ is a weak solution of \eqref{1} in $\overline{\Omega}\setminus (\Gamma \cup \Sigma)$ if
\begin{equation} \label{101}
\begin{aligned}
\int_{\Omega} \left(  \mathscr{A}_p ( \cdot , \grad u) \bullet \deriv \zeta     
  +    f(\cdot , u) \zeta \right) \, \din \wghtv  + \int_{\partial \Omega} b(\cdot , u)   \zeta \, \ds = 0,
\end{aligned}
\end{equation}
holds for all $\zeta \in W_{\operatorname{loc}}^{1,p(x)}(\overline{\Omega} \setminus ( \Gamma \cup \Sigma ); \wghtv )\cap L_{\operatorname{loc}}^{\infty}(\overline{\Omega}\setminus ( \Gamma \cup \Sigma ) )$ such that $\operatorname{supp}\zeta \subset \overline{\Omega}\setminus ( \Gamma\cup \Sigma)$.
\end{definition}
\begin{definition} \label{299}
 We say that a weak solution $u$ of \eqref{1} in $\overline{\Omega}\setminus (\Gamma \cup \Sigma)$ has a removable singularity at $\Gamma$ if $u \in W^{1,p(x)} _{\loc}(\overline \Omega \setminus \Sigma ; \wghtv )\cap L^{\infty}_{\loc}(\overline \Omega \setminus \Sigma  )$ and \eqref{101} holds for all $\zeta \in W^{1,p(x)} _{\loc }(\overline \Omega \setminus \Sigma ; \wghtv )\cap L^{\infty} _{\loc}(\overline \Omega \setminus \Sigma)$ such that $\spt \zeta\subset \overline \Omega \setminus \Sigma$.
 \end{definition}

We work under the following assumptions.

\medskip

\begin{enumerate}[label=($A_{\arabic*}$)]
\item \label{102} For a weak solution $u$ to \eqref{1}, we suppose that there exist $\boldsymbol{\delta} >d$ and $\ctb _0>0$ such that, for all $r_1, r_2 \in (0,1)$ with $r_1<r_2$, we have
$$
u^{\pm}(x)\leq \ctb _0 m^{\pm}_{r_1,r_2} \left(1 + \pid (x) \right)^{-\boldsymbol{\delta}} \quad \text { a.e. in } \quad \left\{x\in \Omega \mid r_1< \piu (x)  <r_2 \right\},
$$
where $m^{\pm}_{r_1,r_2}=\operatorname{ess} \operatorname{sup} \{u^{\pm}(x)\mid x\in \Omega, \,  r_1< \piu (x) <r_2\}$, $u^+ = \max\{u,0\}$, and $u^- = -\min\{u,0\}$.

 \item \label{104} Suppose there exist constants  $\ctb_1, \ctb_2>1$ such that, for all $0<r\le 1$,
$$
\int_{\{x\in\Omega \mid 0<\piu(x)<r\}} (1+\pid(x))^{-\tfrac{\boldsymbol{\delta} }{2}\min \{(p^- -1)p^- ,1 \}} \, \din \wghtv
\le \ctb_1 r^{n_d},
$$
and for all $r_1\in(0,1]$, $r_2>0$, and $0\le \alpha<n_d$,
\begin{gather*}
\int _{\{ x \in \Omega \mid 0< \piu (x) <r_1, \  \pid (x) <r_2 \}}  \piu  ^{-\alpha} \, \din \wghtv \leq  \frac{\ctb _2 }{n_d -\alpha} r_1^{n_d-\alpha} r_2 ^d .
 \end{gather*}
 
 \item \label{106} 

There exist constants   $\gaglia$ and $\consgaglia>0$ such that $\overline{\mathcal U}\setminus\partial\Omega\subset\Omega$, where
\[
\mathcal U:=\bigl\{x\in\Omega \mid  \piu (x)<2\rdo\bigr\},
\qquad
p^-<\gaglia\le \frac{n p^-}{n-p^-}.
\]
Moreover, for every $v\in W^{1,p^-}(\Omega ; \wghtv )$ with compact support in $\overline{\mathcal U}\setminus(\partial\mathcal U\cap\Omega)$, we have
\begin{equation}\label{107}
\left(\int_{\Omega} |v|^{\gaglia}\,\din \wghtv \right)^{\frac{1}{\gaglia}}
\le \consgaglia
\left(\int_{\Omega} |\grad v|_{\fnf}^{p^-}\, \din \wghtv \right)^{\frac{1}{p^-}}.
\end{equation}

\end{enumerate}

\medskip

\begin{remark}

 In the case $\Omega =   \mathbb{R}^d \times B_1^{n-d,+}(0)$, the proof of \eqref{107} follow from Gagliardo-Nirenberg-Sobolev inequality in  \cite{zbMATH05681750}, by extending function $u : \Omega \to \mathbb{R}$ to 
\begin{equation*}
\hat u (x) = \left\{ \begin{aligned}
& u(x) & & \text { if } x \in \Omega,\\
& u((x_1, \ldots , x_{n-1} , -x_n)) & & \text { if } x \in\mathbb{R}^d \times B_1^{n-d}(0) \setminus   \Omega .
\end{aligned}\right.
\end{equation*}

We also have
$$
\gaglia = \frac{np^-}{n-p^-}, \qquad  \consgaglia = \frac{2^{\frac{1}{n}}p^-(n-1)}{n-p^-}.
$$


\end{remark}


Our first main result is  stated as follows:
\begin{theorem}\label{23}
Under the assumptions \ref{fnf16}, \ref{fnf21}, \ref{43}, \eqref{224}--\eqref{243}, and \ref{102}--\ref{106}, if 
$u\in W^{1,p(x)} _{\operatorname{loc}}(\overline{\Omega} \setminus ( \Gamma \cup \Sigma ); \wghtv ) \cap L^\infty _{\operatorname{loc}} (\overline{\Omega} \setminus (\Gamma \cup \Sigma) )$
is a weak solution of \eqref{1} in $\overline{\Omega} \setminus (\Gamma \cup \Sigma)$, then the singularity of $u$ at $\Gamma$ is removable.
\end{theorem}

The following result follows from Lemma~\ref{219} and Corollary~\ref{277} by taking $\ell = \tfrac{4}{3}r$ and $r=\piu (x)$.
\begin{proposition} \label{266}
Under the assumptions \ref{fnf16}, \ref{fnf21}, \ref{43}, \eqref{224}, \eqref{225}, \ref{102}--\ref{106}. Suppose that $u \in W_{\operatorname{loc}  }^{1, p (x) }(\overline{\Omega} \setminus ( \Gamma \cup \Sigma ); \wghtv )\cap L_{\operatorname{loc}  }^{\infty}(\overline \Omega  \setminus (\Gamma \cup \Sigma))$ is a solution of \eqref{1} in $\overline \Omega  \setminus (\Gamma \cup \Sigma)$. Then,  
\begin{equation}\label{275}
|u(x)| \leq {\bf C}  \piu (x)    ^{-\tau}, \quad \text { \textit{a.e.} in } \quad  \left\{x \in \Omega   \mid 0< \piu (x)    <\frac{3\rdo}{4} \right\}  ,
\end{equation}
where 
\begin{align*}
\tau =  \displaystyle \frac{1}{   C_1}\left[  p^+  - \frac{\epsilon n_d(\gaglia -p^{-})}{\gaglia }  \right] \geq \frac{p^+}{q - p^+ +1},
\end{align*}
with $ C_1 = (  q      - p^+   +1)-(  q      - p^-   +1)\epsilon$ and $\epsilon \in ( 0, \tfrac{q      - p^+   +1}{q      - p^-   +1})$. Also,
$$
\begin{aligned}
{\bf  C }={\bf c_1} ^{\frac{8(\beta + q)  n}{\min \{   C_1 , 1\} ^2}  } \left[   C _2  (\consgaglia +1)^{n}\frac{\gaglia+p^{-} (\gaglia -1) \epsilon + \gaglia (1-\epsilon)q  }{( \gaglia - p^- )\epsilon} \right]^{\frac{2\beta}{  C_1} },
\end{aligned}
$$
 with $\beta   =  1 +p^- \frac{\gaglia -  1}{\gaglia   }\epsilon $ and $  C_2 =  C_2 (\Lambda _1 , \Lambda _2 , \ctb _1 , \revc _{\fnf} , n_d , n)>1$, where \({\bf c_1}\) is given in Lemma \ref{64}.
\end{proposition}

Following the same techniques used in Theorem~\ref{23} and Proposition~\ref{266}, we
obtain analogous results when the singular set lies in the interior of the domain.
\begin{corollary} 
[Interior singularities]  \label{61}  Assume that
 $$\Gamma = \piu ^{-1} (\{0\}) \subset \Omega$$ 
 and 
 $$\mathcal{U} =\{x\in \Omega \mid \piu (x)< 2\rdo\} \Subset \Omega.
 $$
   Under the assumptions \ref{fnf16}, \ref{fnf21}, \ref{43}, \eqref{224}--\eqref{243}, \ref{102}--\ref{106}, suppose that  $u \in W_{\operatorname{loc}  }^{1, p(x)  }(\Omega   \setminus \Gamma ; \wghtv ) \cap L_{\operatorname{loc}  }^{\infty}(\Omega  \setminus \Gamma)$ is  a weak solution  in $\Omega   \setminus \Gamma$ of the equation
\begin{equation}\label{271}
- \operatorname{div}[\mathscr{A}_p(\cdot , \grad u )]  +   f(\cdot , u) + \widehat b(\cdot , u) = 0 \quad \text { in } \Omega \setminus \Gamma.
\end{equation} 
Here $\widehat b:\Omega\times \mathbb{R}\to \mathbb{R}$ is measurable and locally bounded, and satisfies $\widehat b(x,u)\operatorname{sign}u\geq 0$ for a.e. $x\in \Omega$ and all $u\in \mathbb{R}$.

Then the singularity of $u$ at $\Gamma$ is removable.
\end{corollary}

\begin{corollary} 
 [Interior pointwise estimate]     \label{223} Under the assumptions \ref{fnf16}, \ref{fnf21}, \ref{43}, \eqref{224}, \eqref{225}, \ref{102}--\ref{106}. Suppose that $u \in W_{\operatorname{loc}  }^{1  , p (x)  }( \Omega  \setminus \Gamma ; \wghtv ) \cap L_{\operatorname{loc}  }^{\infty}( \Omega  \setminus \Gamma)$ is a weak solution  in $ \Omega  \setminus \Gamma$ of equation \eqref{271}. Then, 
$$
|u(x)| \leq {\bf C} \piu (x)  ^{-\tau} \quad\text{\textit{a.e.} in}\quad \left\{x \in \Omega \mid  0< \piu (x)  <\frac{3\rdo}{4} \right\},
$$
where the constants ${\bf C}$ and $\tau$  are as in Proposition \ref{266}.
\end{corollary}


\medskip

\noindent {\it The limit problem as \(p^+\to 1\).}  In this subsection, we assume that $(\rmf, \fnf, \wghtv)$ is a complete, reversible, $n$-dimensional Finsler manifold. We further assume that $\wghtv$ is doubling and that $\rmf$ satisfies a weak Poincaré inequality; see Definitions \ref{310} and \ref{311}, and Remark \ref{312}.
 In addition, we suppose that for every $s \in (1,\min\{2,q+1\})$ there exist constants $\mathtt{C}_{s} > 0$ and $\gamma_{s} > 0$ such that 
\begin{equation*}
s < \gamma_{s} \leq \frac{ns}{n-s}, 
\qquad 
\sup_{s \in (1,\min\{2,q+1\})} \mathtt{C}_{s} < \infty,
\end{equation*}
and for every $v\in W^{1,s} (\Omega ; \wghtv)$ with $\spt v \Subset \Omega$, we have
\begin{equation*}
\left(\int_{\Omega} |v|^{\gamma_{s}}\,\din \wghtv \right)^{\frac{1}{\gamma_{s}}}
\le \mathtt{C}_{s}
\left(\int_{\Omega} |\grad v|_{\fnf}^{s}\, \din \wghtv \right)^{\frac{1}{s}}.
\end{equation*}

We also assume that
\begin{equation}\label{291}
n>  \frac{q+1}{q}, \qquad q>0.
\end{equation}

We work with a domain $\mathscr{U} \Subset \Omega$, and for every variable exponent with $1< p^- \le p^+ <\min\{2,q+1\}$, we assume the existence of weak solutions $u _p \in W^{1, p(x)}_{\loc}(\Omega ; \wghtv) \cap L^{\infty} _{\loc}(\Omega)$ to
\begin{equation} \label{247}
-\operatorname{div} \left(  |\grad u_p|_{\fnf} ^{p-2} \grad u_p \right)  + |u_p|^{q-1}u_p = 0 
\quad \text{in } \Omega,
\end{equation}
see Definitions \ref{298} and \ref{299}. In particular, we have
\begin{equation} \label{248}
\begin{aligned}
\int_{\mathscr{U}}  |\grad u_p|_{\fnf} ^{p-2 }\grad u_p \bullet \deriv  \zeta   \, \din \wghtv    
  +  \int _{\mathscr{U}}  |u_p|^{q-1}u_p \zeta \, \din \wghtv    = 0,
\end{aligned}
\end{equation}
for all $\zeta \in W^{1, p(x) }_0 (\mathscr{U} ; \wghtv) \cap L^{\infty}(\mathscr{U})$.

We study the existence of a solution to
\begin{equation}\label{288}
 -\Delta_1 u + |u|^{q-1}u= 0 \quad \text { in } \mathscr{U}.
\end{equation}

\begin{definition}
A function $u \in B V(\mathscr{U} ;  \wghtv)$ is a solution to \eqref{288} if there exists $\mathbf{z} \in L^{\infty} (T\mathscr{U})$ with $\||\mathbf{z}|_{\fnf}\|_{L ^\infty (\mathscr{U})} \leq 1$ such that
$$
\begin{aligned}
-\operatorname{div} \mathbf{z}  + |u|^{q-1} u & =0 \quad \text { in } \mathcal{D} ^\prime (\mathscr{U}), \\
(\mathbf{z}, D u) & =\|D u\|_{v} \quad \text { as measures in } \mathscr{U}.
\end{aligned}
$$
\end{definition}

Here, for $v\in BV(\mathscr{U} ; \wghtv)\cap L^{q+1} (\mathscr{U} ; \wghtv)$ and for a bounded vector field ${\bf z}$ such that $\operatorname{div} {\bf z}  \in L^{\frac{q+1}{q}}(\mathscr{U} ; \wghtv)$, we define $({\bf z} , Dv) : C^1_c(\mathscr{U}) \to \mathbb{R}$ by
$$
\langle ({\bf z} , Dv) , \zeta \rangle := -\int _{\mathscr{U}} v\zeta \operatorname{div} {\bf z} \, \din \wghtv - \int _{\mathscr{U}} v{\bf z} \bullet \deriv \zeta \, \din \wghtv.
$$

We show that there exists a subsequence $\{u_{p_m}\}$ with $p^+_m \to 1$, ${\bf z} \in L^{\infty} (T\mathscr{U})$, and $u\in BV (\mathscr{U} ; \wghtv) \cap L^{q+1} (\mathscr{U} ; \wghtv )$, such that
\begin{gather*}
|\grad u_{p_m}|_{\fnf} ^{p_m-2} \grad u_{p_m} \rightharpoonup \mathbf{z} 
\quad \text { weakly in } L^s ( T\mathscr{U} ; \wghtv ) \text { for every } 1 \leq s<\infty, \\
u_{p_m} \to u \quad \text { strongly in } L^{1}( \mathscr{U} ; \wghtv ).
\end{gather*}
See Lemma \ref{283} for further details. In that lemma, Corollary \ref{223} plays an essential role.

Through this sequence of arguments, following the proof of
\cite[Theorem 4.3]{zbMATH07751769}, which deals with Sobolev spaces
with constant exponent on domains in $\mathbb{R}^n$, we obtain the
following result.

\begin{proposition}\label{302}
Assume that, for every variable exponent satisfying $1 < p^- \le p^+ < \min\{2,q+1\}$,
there exists a weak solution
$u_p \in W^{1,p(x)} _{\loc}(\Omega;\wghtv) \cap L^{\infty}_{\loc}(\Omega)$
of \eqref{247}. Then there exists a solution of the limit problem    \eqref{288}.
\end{proposition}




\section{Preliminaries} \label{31}

\subsection{Finsler Manifolds}

We recall some facts and notation about Finsler manifolds. Most properties can be found in \cite{farkas2015singular, mester2022functionalinequa,  ohta2009heatfinsler, zbMATH01624431, zbMATH05118101, xiong2020uniquenessnonnegative}.

Throughout this paper, $\fnf : T\rmf \to [0, \infty)$ is a function, called a Finsler structure, satisfying the following properties:
\begin{enumerate}[label=($F_{\arabic*}$)]
\item $\fnf$ is $C^{\infty}( T\rmf \setminus \{0\} )$.

\item $\fnf (x , t Y)=t \fnf ( x , Y)$ for all $(x, Y) \in T \rmf$ and $t> 0$.

\item \label{304} For each $x_0\in \rmf$, there is a chart $\psi:\widetilde U\subset\mathbb{R}^n\to \psi(\widetilde U)\subset \rmf$ such that, for $x=\psi(\mathbf{x})$ and $Y\in T_x\rmf\setminus\{0\}$, the matrix defined by
$$
g_{ij}(x,Y):=\left.\frac{\partial^2}{\partial \xi^i\partial \xi^j}
\Big(\tfrac12\,\widetilde F(\mathbf{x},\xi)^2\Big)\right|_{\mathbf{x}=\psi^{-1}(x),\ \xi=d\psi_{\mathbf{x}}^{-1}(Y)}
$$
is positive definite, where
$$
\widetilde \fnf (\mathbf{x},\xi):=\fnf\bigl(\psi(\mathbf{x}),d\psi_{\mathbf{x}}(\xi)\bigr).
$$
More precisely, there exist constants $\kappa_1,\kappa_2>0$ such that
$$
\kappa_1\sum_{k=1}^n(\eta^k)^2
\le g_{ij}(x,Y)\,\eta^i\eta^j
\le \kappa_2\sum_{k=1}^n(\eta^k)^2
$$
for all $x\in\psi(\widetilde U)$, $Y\in T_x\rmf\setminus\{0\}$, and $\eta=\eta^i\partial_i\in T_x\rmf$, where \(\{\partial_i\}_{i=1}^n\) is the basis associated with the chart \(\psi\).
\end{enumerate}

If $\fnf(x,Y)=\fnf(x,-Y)$ for all $(x,Y)\in T\rmf$, then the Finsler manifold $(\rmf,\fnf)$ is said to be {\it reversible}.

Let $V=v^i\partial_i$ be a non-vanishing vector field on an open subset $U\subset \rmf$. Then one can define a Riemannian metric $g_V$ on $TU$ by
$$
g_V(X,Y)(x):=X^iY^j\,g_{ij}(x,V(x)).
$$

Throughout this paper, we fix a measure $\wghtv$ on $\rmf$. For each point $x\in \Omega$, we assume that there exist a neighborhood $U$, a local coordinate system $\psi:\widetilde U\to U$, and a measurable function $\tilde{\wghtv}:\widetilde U\to \mathbb{R}^+$ such that, for some constant $C>0$,
\begin{equation}\label{297}
C^{-1}\le \tilde{\wghtv}\le C \qquad \text{in } \widetilde U.
\end{equation}

Moreover, for every measurable set $E\subset U$, we have
\begin{equation}\label{256}
\wghtv(E)=\int_{\psi^{-1}(E)} \tilde{\wghtv}\,\din {\bf x}.
\end{equation}

As an immediate consequence, for every integrable function $f:U\to \mathbb{R}$,
\begin{equation}\label{255}
\int_E f\,\din \wghtv
=
\int_{\psi^{-1}(E)} (f\circ \psi)\,\tilde{\wghtv}\,\din {\bf x}.
\end{equation}

\subsubsection{The Legendre Transform}
The polar transform (or co-metric) \( \fnfs : T^\ast \rmf \to [0,\infty) \) is the dual metric associated with \( \fnf \), defined by
\begin{equation}\label{162}
\fnfs(x,\omega)=\sup_{v\in T_x\rmf\setminus\{0\}} \frac{\omega(v)}{\fnf(x,v)},
\end{equation}
where \(T^\ast \rmf=\bigcup_{x\in\rmf} T_x^\ast \rmf\) denotes the cotangent bundle of \( \rmf \), and \(T_x^\ast\rmf\) is the dual space of \(T_x\rmf\).

Now let \(x=\psi({\bf x})\in \psi(\widetilde U)\subset \rmf\), where \(\psi\) is a chart. Then, for every \(\omega\in T_x^\ast\rmf\), there exists a unique \((\omega_1,\ldots,\omega_n)\in\mathbb{R}^n\) such that
\[
\omega\bigl(d\psi_{\bf x}(v)\bigr)=\bigl\langle (\omega_1,\ldots,\omega_n),v\bigr\rangle
\qquad \forall\, v\in\mathbb{R}^n,
\]
where \(\langle\cdot,\cdot\rangle\) denotes the Euclidean inner product. In these local coordinates, we write
\[
\widetilde{\fnfs}\bigl({\bf x},(\sigma^1,\ldots,\sigma^n)\bigr)
=
\fnfs\bigl(x,\sigma^i \partial_i^\ast(x)\bigr),
\qquad x=\psi({\bf x}),
\]
where \(\partial_i^\ast(x)\in T_x^\ast \rmf\) is defined by
\[
\partial_i^\ast(x)(\partial_j(x)):=\delta_{ij}.
\]

Since \( \fnfs(x,\cdot)^2 \) is twice differentiable on \(T_x^\ast\rmf\setminus\{0\}\), one can define the dual Hessian matrix \( [g_{ij}^\ast(x,\omega)] \) by
\[
g_{ij}^\ast(x,\omega):=
\left.
\frac{1}{2}
\frac{\partial^2}{\partial \sigma^i\,\partial \sigma ^j}
\widetilde{\fnfs}\bigl({\bf x},(\sigma ^1,\ldots,\sigma ^n)\bigr)^2
\right|_{
{\bf x}=\psi^{-1}(x),\,
(\sigma^1,\ldots,\sigma^n)=(\omega_1,\ldots,\omega_n)
} .
\]

As in property \ref{304}, for every \(x\in \rmf\) there exist a chart (still denoted by \(\psi\)) and positive constants \(\kappa_1^\ast\) and \(\kappa_2^\ast\) such that
\begin{equation}\label{308}
\kappa_1^\ast \sum_{k=1}^n (\sigma^k)^2
\le
g_{ij}^\ast(x,\omega)\,\sigma^i\sigma^j
\le
\kappa_2^\ast \sum_{k=1}^n (\sigma^k)^2 .
\end{equation}
This follows from \cite[Lemma 1.1]{ohta2009heatfinsler}.

The Legendre transform \(J^\ast:T^\ast\rmf\to T\rmf\) is defined as follows: for each fixed \(x\in\rmf\), \(J^\ast\) associates to each \(\omega\in T_x^\ast\rmf\) the pair \((x,v)\in T\rmf\), where \(v\in T_x\rmf\) is the unique maximizer of the map
\begin{equation}\label{228}
v\mapsto \omega(v)-\frac{1}{2}\fnf^2(x,v).
\end{equation}
See \cite[§ 3.2 Existence of minimizer]{arXiv:2001.00216} or \cite[Corollary 25.15]{zbMATH00045061}. The strict convexity of \(\tfrac12\fnf^2(x,\cdot)\) implies the uniqueness of the maximizer in \eqref{228}.

In connection with \eqref{228}, \cite[Lemma 5.8 (Fenchel--Young)]{arXiv:2001.00216} (see also \cite[Proposition 51.2]{zbMATH03933858}) states that if \(f:X\to\overline{\mathbb{R}}\) is proper, convex, and lower semicontinuous on a normed vector space \(X\), then for any \(x\in X\) and \(x^\ast\in X^\ast\), the following statements are equivalent:
\begin{enumerate}[label=$(\roman*)$]
\item \(\langle x^\ast,x\rangle=f(x)+f^\ast(x^\ast)\);
\item \(x^\ast\in\partial f(x)\);
\item \(x\in\partial f^\ast(x^\ast)\).
\end{enumerate}
Here \(f^\ast\) denotes the Fenchel conjugate (or convex conjugate) of \(f\), defined by
\[
f^\ast:X^\ast\to\overline{\mathbb{R}},\qquad
f^\ast(x^\ast)=\sup_{x\in X}\bigl\{\langle x^\ast,x\rangle-f(x)\bigr\}.
\]
Moreover, the (convex) subdifferential of \(f\) at \(x\in \operatorname{dom}f\) is defined by
\[
\partial f(x):=
\left\{
x^\ast\in X^\ast \,\middle|\,
\langle x^\ast,\tilde x-x\rangle\le f(\tilde x)-f(x)
\quad \text{for all } \tilde x\in X
\right\}.
\]

Therefore, if $J^\ast(x,\omega)=(x,{\bf v})$, then
$$
\omega \in \partial \left[\frac{1}{2}\fnf^2(x,{\bf v})\right].
$$

Moreover, if  $\omega \not \equiv 0$, then
$$
\omega({\bf v})-\frac{1}{2}\fnf^2(x,{\bf v})
\ge t\omega(Y)-\frac{1}{2}t^2\fnf^2(x,Y),
$$
for all $t\in\mathbb{R}$ and all $Y\in T_x\rmf$ such that $\omega(Y)>  0$ and $\fnf(x,Y)>0$.

Optimizing the right-hand side with respect to $t$, we obtain
$$
\omega({\bf v})-\frac{1}{2}\fnf^2(x,{\bf v})
\ge \frac{1}{2}\frac{\omega(Y)^2}{\fnf^2(x,Y)}.
$$
Hence,
$$
2\left(\omega({\bf v})-\frac{\fnf(x,{\bf v})\omega(Y)}{\fnf(x,Y)}\right)
\ge
\left(\fnf(x,{\bf v})-\frac{\omega(Y)}{\fnf(x,Y)}\right)^2.
$$
It follows that
$$
\fnfs(x,\omega)=\frac{\omega({\bf v})}{\fnf(x,{\bf v})},
\qquad
\fnf^2(x,{\bf v})=\omega({\bf v}).
$$
In particular,
$$
\fnf(x,{\bf v})=\fnfs(x,\omega).
$$

In local coordinates, if $J^\ast(x,\omega)=(x,{\bf v})$, then, by \cite[Section 1.2]{ohta2009heatfinsler}, the fiber component ${\bf v}$ is given by
$$
{\bf v} = \sum_{i=1}^n
\frac{\partial}{\partial \omega_i}
\left(\frac{1}{2}{\fnfs}^2(x,\omega)\right)\partial_i,
$$
where $\{\partial_i\}_{i=1}^n$ denotes the basis associated with the chart $\psi$.

\subsubsection{Gradient and Distance Functions}

For a weakly differentiable function \( u: \rmf \rightarrow \mathbb{R} \), the gradient vector at \( x \) is defined by
$$
\grad u(x):=J^*\bigl(x, \deriv u(x)\bigr),
$$
for every regular point \( x \in \rmf \), where the derivative \( \deriv u(x) \in T_{x}^\ast \rmf \) is well-defined. By applying the properties of the Legendre transform, it follows that
\begin{equation} \label{fnf113}
 \fnfs  \bigl(x , \deriv v (x) \bigr)=\fnf  \bigl(x , \grad v (x) \bigr)
\end{equation}

In a local coordinate system, if \( x \in \psi(\tilde{U}) \), we have 
$$
\deriv u(x)=\sum_{i=1}^n \bigl( \partial _i (u)  \partial _i ^\ast \bigr)(x) \quad \text { and } \quad \grad u(x)=\sum_{i, j=1}^n \left(  g_{i j}^\ast (\cdot  , \deriv u) \partial _i(u) \partial _j \right)(x) ,
$$
where $\partial ^\ast _i (x) \in T^\ast _{x} \rmf$ is defined by $\partial ^\ast _i (x)(\partial _j (x)):= \delta _{ij}$.

For a differentiable vector field $Y : \rmf \rightarrow T\rmf$ on $\rmf$, we define its divergence $\operatorname{div} Y: \rmf \to \mathbb{R}$ through the identity
$$
\int_{\rmf} u \operatorname{div}  Y \, \din \wghtv =-\int_{\rmf} \deriv u (Y) \, \din \wghtv.
$$

Let $\Omega _0 $ be a compact domain with smooth boundary $\partial \Omega  _0 $ and $\nu$ denote the outward pointing normal vector. Then, for any smooth vector field $X$ on $\rmf$,
$$
\int_{\Omega_0} \operatorname{div}(X) \, \din  \wghtv=\int_{\partial \Omega_0} g_{\nu}( \nu, X) \, \din A_{\wghtv},
$$
where $\din A_{\wghtv}$ is the volume form on $\partial \Omega _0$ induced from $\din \wghtv$.

We observe that the nonlinearity of the Legendre transform extends to the gradient vector, namely \(\grad (u+v) \neq \grad u + \grad v\) in general. For the same reason, at points \(x\) where \(\grad u(x) = 0\), the gradient vector field \(\grad u\) is, in general, not differentiable (even if \(u\) is smooth), but only continuous.




We define the distance function $d : \rmf \times \rmf \rightarrow [0,\infty)$ by
$$
d(x , y)=\inf _\gamma \int_a^b \fnf  \bigl(  \gamma (t) , \gamma^\prime (t)  \bigr) \, \din t,
$$
where the infimum is taken over all piecewise differentiable curves \(\gamma : [a,b] \to \rmf\) with \(\gamma(a) = x\) and \(\gamma(b) = y\). This definition is equivalent to
$$
d(x, y):=\sup \Bigl\{u(y)-u(x) \  |  \ u \in C^1(\rmf), \fnf ( z, \grad u(z)) \leq 1 \text { for all } z \in \rmf \Bigr\} .
$$

 For a fixed \(y \in \rmf\), the distance function \(x \mapsto d(y, x)\) satisfies \(\fnf(x, \grad d(y, x)) = 1\) for almost every \(x \in \rmf\). Moreover, the distance function \(d\) possesses the following properties of a metric:
\begin{enumerate}
\item[$(a)$] \(d(x, y) \geq 0\) for all \(x, y \in \rmf\), and \(d(x, y) = 0\) if and only if \(x = y\).

\item[$(b)$] \(d(x, z) \leq d(x, y) + d(y, z)\) for all \(x, y, z\in \rmf\).
\end{enumerate}

However, in general, the distance function is not symmetric. In fact, \(d(x, y) = d(y , x)\) for all \(x, y \in \rmf\) if and only if \((\rmf , \fnf)\) is a reversible Finsler manifold.

\begin{lemma}\label{307}
Let $u : \rmf \to \mathbb{R}$ be differentiable at $x$. Then
\begin{equation}
\fnf(x, \grad u(x)) = \underset{y\to x}{\limsup}\,\frac{u(y)-u(x)}{d(x,y)}.
\end{equation}
\end{lemma}

\begin{proof}
Since \(u\) is differentiable at \(x\), for every \(\epsilon>0\) there exists  a sufficiently small $\delta \in (0,1)$ such that
\[
\frac{u(\exp_x(v^i\partial_i(x))) - u(x)}{|v|_{\mathbb{R}^n}}
< \deriv u_x\bigl(v^i\partial_i(x)\bigr) + \epsilon,
\qquad 0 < |v|_{\mathbb{R}^n} < \delta < 1.
\]

Hence,
\[
\frac{u(\exp_x(v^i\partial_i(x))) - u(x)}{d\bigl(x,\exp_x(v^i\partial_i(x))\bigr)}
<
|v|_{\mathbb{R}^n}\,\frac{\deriv u_x\bigl(v^i\partial_i(x)\bigr)}{\fnf\bigl(x,v^i\partial_i(x)\bigr)}
+
\frac{|v|_{\mathbb{R}^n}}{\fnf\bigl(x,v^i\partial_i(x)\bigr)}\,\epsilon,
\]
because
\[
d\bigl(x,\exp_x(v^i\partial_i(x))\bigr)=\fnf\bigl(x,v^i\partial_i(x)\bigr).
\]

Therefore, we obtain
\begin{equation}\label{305}
\limsup_{y \to x} \frac{u(y)-u(x)}{d(x,y)}
\leq \fnfs(x,\deriv u(x)) = \fnf(x,\grad u(x)).
\end{equation}

Furthermore, for $v \in T_x\rmf \setminus \{0\}$, we have
\[
\limsup_{y \to x} \frac{u(y)-u(x)}{d(x,y)}
\geq
\lim_{t\downarrow 0} \frac{u(\exp_x(tv)) - u(x)}{d(x,\exp_x(tv))}
=
\frac{\deriv u_x(v)}{\fnf(x,v)}.
\]
Which implies
\begin{equation}\label{306}
\limsup_{y \to x} \frac{u(y)-u(x)}{d(x,y)}
\geq \fnfs(x,\deriv u(x)).
\end{equation}

From \eqref{305} and \eqref{306}, we conclude the proof of the lemma.
\end{proof}

\subsection{Variable exponent Sobolev spaces}

Next, we recall some key facts and notation regarding variable exponent Lebesgue and Sobolev spaces. A detailed discussion of the properties of these spaces can be found in \cite{gaczkowski2013sobolev, gaczkowski2016sobolev} for Riemannian manifolds, \cite{futamura2006sobolev, zbMATH06473105, harjulehto2006sobolev, harjulehto2006variablemetric} for metric spaces, and \cite{diening2011lebesgue, kovuavcik1991spaces} for the Euclidean setting. For Finsler manifolds with constant exponent \(p\), we refer the reader to \cite{farkas2015singular, arXiv:2409.05497, ohta2009heatfinsler}.

Let $p : \Omega \rightarrow [1,\infty)$ be a measurable function, and let $\hat \Omega  \subset \Omega$ be an open subset. The variable exponent, or generalized, Lebesgue space $L^{p(x)}(\hat \Omega ; \wghtv)$ is the space of all measurable functions \( u : \hat \Omega \rightarrow \mathbb{R} \) for which the functional
$$
\rho_{p  }(u):=\int_{\hat \Omega}|u|^{p} \, \din \wghtv
$$
is finite. This is a special case of an Orlicz-Musielak space; see \cite{musielak2006orlicz}.

The functional \( \rho_{p} \) is convex and is sometimes called a convex modular \cite{kovuavcik1991spaces}. The space \( L^{p(x)}(\hat \Omega; \wghtv) \) is a Banach space with respect to the Luxemburg-Minkowski-type norm
$$
\|u\|_{L^{p(x)} ( \hat \Omega ; \wghtv)}:=\inf  \left\{t>0  \mid \rho_{p } \left(\frac{u}{t} \right) \leq 1  \right\} .
$$
By \cite[Lemma 3.1]{zbMATH02160501},  $L^{p(x)} (\hat \Omega ; \wghtv)$ is a Banach space. Moreover, by \cite[Proposition 2.3]{harjulehto2006variablemetric}, it is reflexive if $1<\inf _{\hat{\Omega}} p \leq \sup _{\hat{\Omega}} p <\infty$.


In what follows, we will use the generalized Hölder inequality:
\begin{equation}\label{fnf116}
\int_{\hat \Omega }  |uv| \, \din \wghtv \leq 2 \|u\|_{L^{p(x)} ( \hat \Omega ; \wghtv) } \|v\|_{L^{p^{\prime}(x)} ( \hat \Omega ; \wghtv)},
\end{equation}
as stated in \cite[Theorem 2.1]{kovuavcik1991spaces}.

To compare the functionals $\|\cdot \|_{L^{p(x)} ( \hat \Omega ; \wghtv)}$ and $\rho_{p}$, we have the following relations:
\begin{equation}\label{152}
\min  \left\{\rho_{p}(u)^{1 / p^{-}}, \rho_{p}(u)^{1 / p^{+}} \right\} \leq\|u\|_{L^{p(x)} ( \hat \Omega ; \wghtv)} \leq \max  \left\{\rho_{p} (u) ^{1 / p^{-}}, \rho_{p}(u)^{1 / p^{+}} \right\},
\end{equation}
see \cite[Corollary 2.22]{zbMATH06090416}.


The variable exponent Sobolev space $W^{1, p(x)}(\hat \Omega ; \wghtv)$ is defined by
$$
W ^ {1, p(x)}(\hat \Omega ; \wghtv)  :=\left\{u \in L^{p(x)} (\hat \Omega ; \wghtv) \  |  \ \int_{\hat \Omega} \fnf (\cdot , \grad u)^p \, \din \wghtv<\infty \right\} .
$$
This space is a Banach space when equipped with the norm
\begin{equation*}\label{fnf177}
\|u\|_{W^{1, p(x)} (\hat \Omega ; \wghtv)}:= \|u\|_{ L^{p(x)} ( \hat \Omega ; \wghtv) } + \left\|\fnf ( \cdot ,  \grad u) \right\|_{L^{p(x)} ( \hat \Omega ; \wghtv)}.
\end{equation*}
We also define $W^{1 , p(x)} _0(\hat \Omega ; \wghtv)$ as the closure of $C^{\infty} _{0} (\hat \Omega)$ in $W^{1 , p(x)} (\hat \Omega ; \wghtv)$ with respect to the norm $\|\cdot\|_{ W^{1, p (x)}( \hat \Omega ; \wghtv )}$.

The tangent module to $( \hat{\Omega} , d,  \wghtv)$, see \cite[Section 7.2]{zbMATH07567821} and \cite{zbMATH07166538}, is defined as
$$
L^{p(x)} (T\hat \Omega ; \wghtv ) := \left\{ V : \hat \Omega \to T \hat \Omega \mid V(x)\in T _x\hat  \Omega, \ \int_{\hat \Omega} \fnf (\cdot , V)^p \, \din \wghtv  <\infty \right\}.
$$

\begin{remark}\label{fnf156} $ $

\begin{enumerate}[label = $(\alph*)$]
\item  \label{fnf157} Following the definition of Sobolev spaces with variable exponent on a Riemannian manifold given in \cite{gaczkowski2013sobolev, gaczkowski2016sobolev}, an alternative way to define a Sobolev space in our setting is as follows. Set
$$
\mathcal{C}_1^{p(x)}(\hat \Omega ; \wghtv )=\left\{u \in C^{\infty}(\hat \Omega) \mid  \fnf \left(\cdot , \grad u\right) \in L^{p(x)}( \hat \Omega ; \wghtv)\right\} .
$$
We then define the variable exponent Sobolev space $H_1^{p(x)}(\hat{\Omega} ; \wghtv )$ as the completion of $\mathcal{C}_1^{p(x)}(\hat \Omega ; \wghtv )$ with respect to the norm $\|\cdot\|_{W^{1,p(x)}( \hat{\Omega};  \wghtv )}$. As in the Euclidean case, in general, $W^{1,p(x)}(\hat{\Omega} ; \wghtv)$ is not equal to $H_1^{p(x)}(\hat{\Omega} ; \wghtv)$; see \cite{zbMATH02196634}.

\item \label{309} From \eqref{308}, we have
$$
\kappa_1^\ast \sum_{i=1}^n (\omega^i)^2 \leq {\fnfs}^2(x,\omega) \leq \kappa_2^\ast \sum_{i=1}^n (\omega^i)^2,
$$
where $\omega(x)=\omega^i \partial_i^\ast(x)\in T_x^\ast \rmf$. Then, if $u$ is differentiable at $x\in\psi (\widetilde U)$,
$$
\sqrt{\kappa_1^\ast}\, |\nabla  (u \circ \psi )|_{\mathbb{R}^n}
\leq \fnf(x,\grad u(x))
\leq \sqrt{\kappa_2^\ast}\, |\nabla  (u \circ \psi)|_{\mathbb{R}^n},
$$
since $\fnfs(x,\deriv u(x))=\fnf(x,\grad u(x))$.
\end{enumerate}
\end{remark}

Note that the classes \( L_{\loc}^{p(x)}(\Omega ; \wghtv) \) and \( W_{\loc}^{1, p(x)}(\Omega; \wghtv) \) depend only on the manifold structure of \( \Omega \) (not on the Finsler structure \( \fnf \) or the measure \( \wghtv \)); see the last paragraph on page 1394 of \cite{ohta2009heatfinsler}. More precisely, let \( \hat \Omega \Subset  \Omega \) be an open set, and consider two measures \( \wghtv_1 \) and \( \wghtv_2 \) satisfying \eqref{297} and \eqref{256}. These measures are comparable on \( \hat \Omega \). In other words, there exists a constant \( C > 0 \) such that
$$
C \wghtv _1 (E) \leq \wghtv _2 (E) \leq C^{-1} \wghtv _1 (E),
$$
for every measurable subset \( E \subset \hat \Omega \). This implies that \( \int_{\hat \Omega} |u|^p  \, \din \wghtv_1 < \infty \) if and only if \( \int_{\hat \Omega} |u|^p \, \din \wghtv_2 < \infty \). Thus, \( L_{\loc}^{p(x)}(\Omega; \wghtv) \) does not depend on the measure \( \wghtv \).

For \( W_{\loc}^{1, p(x)}(\Omega ; \wghtv) \), we further observe that for any two Finsler structures \( \fnf_1 \) and \( \fnf_2 \), by 1-homogeneity and property \ref{304}, we have
$$
c \fnf_1(\cdot, \grad u) \leq \fnf_2(\cdot, \grad u) \leq c^{-1} \fnf_1(\cdot, \grad u) \quad \text { in } \quad \hat \Omega,
$$
for some constant \( c > 0 \). Therefore, \( W_{\loc}^{1, p(x)}(\Omega; \wghtv) \) does not depend on either the measure \( \wghtv \) or the Finsler structure \( \fnf \).


\subsection{BV spaces}

In this section we collect some definitions and properties of the space \(BV\), which can be found, for instance, in \cite{zbMATH06320681, zbMATH07523073, zbMATH07047806, zbMATH07567821, zbMATH08128442, zbMATH05152934}.

By a metric measure space we mean a triple \((\mathbb{X}, d, \mu)\), where \((\mathbb{X}, d)\) is a complete and separable metric space, and \(\mu\) is a nonzero, nonnegative Borel measure on \((\mathbb{X}, d)\) which is finite on bounded sets.

Now we recall the definitions of upper gradient and Poincaré inequality; see \cite{zbMATH05233008, zbMATH08128442, zbMATH06397370} for further details. We say that a curve \(\gamma:[0,1] \to \mathbb{X}\) is {\it absolutely continuous} if there exists a function \(g \in L^1([0,1])\) such that
\[
d(\gamma(t), \gamma(s)) \leq \int_s^t g(r)\,\din r
\]
for all \(s,t \in [0,1]\) with \(s<t\). The minimal function \(g\) with this property is characterized by
\[
|\dot{\gamma}(t)|:=\lim_{h \to 0}\frac{d(\gamma(t+h),\gamma(t))}{h},
\]
and is called the {\it metric speed} of \(\gamma\).
We denote by \(C([0,1],\mathbb{X})\) the space of all continuous curves from \([0,1]\) to \(\mathbb{X}\), equipped with the supremum norm.

We say that a Borel function \(g:\mathbb{X}\to[0,+\infty]\) is an upper gradient of a Borel function \(f:\mathbb{X}\to\mathbb{R}\) if for every absolutely continuous curve \(\gamma:[0,1]\to\mathbb{X}\) we have
\begin{equation*}
|f(\gamma(1))-f(\gamma(0))|
\leq \int_\gamma g\,\din s
:=\int_0^1 g(\gamma(t))|\dot{\gamma}(t)|\,\din t.
\end{equation*}

We denote by \(UG(u)\) the collection of all upper gradients of \(u\).

If \(u_i:\mathbb{X}\to\mathbb{R}\) are continuous and \(g_i \in UG(u_i)\), \(i=1,2\), then:
\begin{enumerate}[label=$(\roman*)$]
\item \(\left|\alpha_1\right|g_1+\left|\alpha_2\right|g_2 \in UG(\alpha_1u_1+\alpha_2u_2)\) for every \(\alpha_1,\alpha_2 \in \mathbb{R}\setminus\{0\}\);
\item \(\left(|u_1|+\epsilon\right)g_2+\left(|u_2|+\epsilon\right)g_1 \in UG(u_1u_2)\) for every \(\epsilon>0\).
\end{enumerate}

 We say that a metric measure space \((\mathbb{X}, d, \mu)\) is a
\textit{Poincaré space} if \(\mu\) is a doubling measure and
\(\mathbb{X}\) supports a weak Poincaré inequality. More precisely:

\begin{definition}\label{310}
We say that a measure \(\mu\) on a metric space \(\mathbb{X}\) is \textnormal{doubling} if there exists a constant \(C_d \geq 1\) such that the following condition holds:
\begin{equation*}
0 < \mu(B(x,2r)) \leq C_d\,\mu(B(x,r)) < \infty
\end{equation*}
for all \(x \in \mathbb{X}\) and \(r > 0\).
\end{definition}

\begin{definition}\label{311}
We say that \(\mathbb{X}\) supports a \textnormal{weak Poincaré inequality} if there exist constants \(C_P > 0\) and \(\lambda \geq 1\) such that for every ball \(B \subset \mathbb{X}\) and every pair \((u,g)\) with \(u \in \operatorname{Lip}_{\mathrm{loc}}(\mathbb{X})\) and \(g \in UG(u)\), there holds
\begin{equation*}
\int_B |u(x)-u_B|\, \din \mu(x)
\leq C_P\, r\int_{\lambda B} g(x)\, \din \mu(x),
\end{equation*}
where \(r\) is the radius of \(B\), and
\[
u_B := \fint_B u \, \din \mu := \frac{1}{\mu(B)} \int_B u \,\din \mu.
\]
\end{definition}

\begin{remark} \label{312} $ $
\begin{enumerate}[label=$(\roman*)$]
\item Let $(M,F,\mu)$ be a forward complete Finsler measure space with finite reversibility $\Lambda$. In \cite[Lemma 4.1]{zbMATH07870791}, a $p$-Poincaré inequality is proved under the condition $\operatorname{Ric}_{\infty}\geq -K$ for some $K\geq 0$:
\[
\int_{B_R}|u-u_{B_R}|^p \, \din \mu
\leq
c_1 e^{c_2\left(K+\delta^2\right)R^2} R^p
\int_{B_{(\Lambda+2)R}} {F^{\ast}}^{p}(\cdot , \deriv u) \, \din \mu,
\qquad
u\in W_{\mathrm{loc}}^{1,p}(M;\mu),
\]
where $p\in(1,\infty)$ and $c_i=c_i(p,n,\Lambda)$, $i=1,2$. Moreover,
\[
\delta:=
\sup_{(x,Y)\in TM\setminus\{0\}}
\frac{|\mathbf{S}(x,Y)|}{F(x,Y)},
\]
where $\mathbf{S}=\mathbf{S}(x,Y)$ denotes the $S$-curvature of $(M,F,\mu)$.

\item Let \((X,d,\mu)\) be a metric measure space, where \(d\) is a metric and
\(\mu\) is a positive complete Borel measure such that
\(0<\mu(B)<\infty\) for every ball \(B\subset X\). Assume that \(X\) is
proper, i.e., every closed bounded subset of \(X\) is compact. Assume
further that \(X\) is connected, that \(\mu\) is locally doubling, and that
\(X\) supports a local \(p\)-Poincaré inequality for some
\(1\leq p<\infty\). Then \cite[Theorem 1.5]{zbMATH06951611} shows that
Lipschitz functions with compact support are dense in \(W^{1,p}(X;\mu)\).
\end{enumerate}
\end{remark}

Given \(u \in \operatorname{Lip}_{\loc}(\mathbb X)\), we define the modulus of the gradient of \(u\) by
\[
\|\nabla u\|(x)=\limsup_{y \to x}\frac{|u(y)-u(x)|}{d(y,x)}.
\]

In the sequel we assume that \((\mathbb{X}, d, \mu)\) is a Poincaré space. Let \(\Omega \subset \mathbb X\) be an open set. Since \(\operatorname{Lip}_{\loc}(\Omega)\) is dense in \(L^1_{\loc}(\Omega ; \mu)\) (see \cite[Theorem 3.3]{zbMATH02160501} and \cite[Theorem 2]{zbMATH01848850}), it makes sense to define, for every \(u \in L^1_{\loc}(\Omega ; \mu)\), the total variation of \(u\) on every open set \(A \subset \Omega\) by
\[
\|Du\|_{v}(A)
=
\inf\left\{
\liminf_{h\to\infty}\int_A \|\nabla u_h\|\,\din\mu
\;\middle|\;
\{u_h\}\subset \operatorname{Lip}_{\loc}(A),\ 
u_h \to u \text{ in } L^1_{\loc}(A ; \mu )
\right\}.
\]

By \cite[Theorem 3.4]{zbMATH05152934}, for every \(u \in L^1_{\loc}(\Omega ; \mu)\), the set function \(\|Du\|_{v}\) is the restriction to the open subsets of \(\mathbb{X}\) of a positive finite measure on \(\mathbb{X}\).

\begin{definition}[Functions of bounded variation]
A function \(u \in L^1_{\loc}(\Omega  ; \mu)\) is said to have locally bounded total variation in \(\Omega\) if \(\|Du\|_{v}(A)<\infty\) for every open subset \(A \Subset \Omega\). A function is said to have bounded total variation in \(\Omega\) if \(\|Du\|_{v}(\Omega)<\infty\).

The vector space of functions with (locally) bounded total variation will be denoted by \(BV(\Omega  ;  \mu)\) (respectively, \(BV_{\loc}(\Omega  ;  \mu)\)).
\end{definition}

The following two propositions can be found in \cite{zbMATH05152934}.
\begin{proposition}[Lower semicontinuity] \label{226}
Let \(\Omega \subset \mathbb X\) be an open set, and let \(\{u_h\}\) be a sequence in \(BV_{\loc}(\Omega ; \mu)\) such that \(u_h \to u\) in \(L^1_{\loc}(\Omega  ;  \mu)\). Then
\[
\|Du\|_{v}(A)\le \liminf_h \|Du_h\|_{v}(A),
\quad \text{for every open set } A \subset \Omega.
\]

In particular, if \(\sup_h \|Du_h\|_{v}(A)<\infty\) for every open set \(A \Subset \Omega\), then the limit function \(u\) belongs to \(BV_{\loc}(\Omega  ;  \mu)\).
\end{proposition}

\begin{proposition}[Compactness] \label{227} 
Let \(\{u_h\} \subset BV_{\loc}(\Omega  ;  \mu)\) be a sequence bounded in \(L^1_{\loc}(\Omega  ;  \mu)\) and satisfying \(\sup_h \|Du_h\|_{v}(A)<\infty\) for every open set \(A \Subset \Omega\). Then there exist \(u \in BV_{\loc}(\Omega  ;  \mu)\) and a subsequence \(\{u_{h_k}\}\) converging to \(u\) in \(L^1_{\loc}(\Omega  ;  \mu)\).
\end{proposition}

The following approximation result is taken from \cite[Lemma 3.2]{zbMATH07852719}; see also the proof of \cite[Lemma 4.1]{zbMATH07567821}. Other references containing similar results include \cite[Corollary 6.7]{zbMATH06855720}.

\begin{lemma}\label{286}
Let \(u \in BV(\Omega;\mu)\). Then there exists a sequence
\(\{u_n\}_{n=1}^{\infty} \subset \operatorname{Lip}(\Omega)\cap BV(\Omega;\mu)\)
such that:
\begin{enumerate}[label=$(\roman*)$]
\item \label{287} \(u_n \to u\) strictly in \(BV(\Omega  ;  \mu)\); that is, \(u_n \to u\) in \(L^1(\Omega  ;  \mu)\) and \(\|Du_n\|_{v}(\Omega)\to \|Du\|_{v}(\Omega)\).

\item If \(p\in[1,\infty)\) and \(u \in L^p(\Omega;\mu)\), then
\(u_n \in L^p(\Omega;\mu)\) for each \(n\), and \(u_n \to u\) in \(L^{p }(\Omega  ;  \mu)\).

\item If \(u \in L^{\infty}(\Omega  ;  \mu)\), then \(u_n \in L^{\infty}(\Omega  ;  \mu)\) and \(u_n \rightharpoonup u\) weakly* in \(L^{\infty}(\Omega  ;  \mu)\).
\end{enumerate}
\end{lemma}

Furthermore, from Proposition \ref{226} and Lemma \ref{286}\ref{287}, together with the Portmanteau Theorem (also known as Alexandrov's Theorem) on weak convergence of measures, see \cite[Theorem 4.4]{zbMATH02123347} or \cite[Theorem 8.2.3]{zbMATH05116074}, we conclude that
\begin{equation}\label{292}
\lim_{n\to\infty}\int_{\Omega}\zeta \|\nabla u_n\|\,\din \mu
=
\int_{\Omega}\zeta\,\din \|Du\|_{v},
\qquad \text{for every } \zeta \in C_c^1(\Omega).
\end{equation}
See \cite[Proposition 4.5.6]{DiMarino2014} for a similar result.

\section{Behavior of solutions near the singular set} \label{32}

In order to prove our main results, we need the following analytic result concerning the behavior of solutions to \eqref{1} near the singular set. We begin with a technical lemma.
\begin{lemma}\label{64}
Let \( 0 < \theta < 1 \), \( \sigma > 0 \), and let
\( \xi : [\tfrac12, 1] \to \mathbb{R} \) be a function satisfying
\[
\xi(t) \ge \xi(\tfrac12) \ge 1 \quad \text{for all } t \in [\tfrac12, 1].
\]
Assume further that
\[
\xi(k) \le {\bf c}_0 (h-k)^{-\sigma} \bigl(\xi(h)\bigr)^{\theta},
\quad \tfrac12 \le k < h \le 1,
\]
for some positive constant \( {\bf c}_0 \).

Let \(1<b\leq 2\) be such that \(1-b\theta>0\). Then there exists a constant \( {\bf c} _1>1\), independent of \(\theta\), \(b\), \(\sigma\), and \(\xi\), such that
\[
\xi(\tfrac12) \le {\bf c} _1^{\frac{\sigma b}{(1-b\theta)(b-1)}}
{\bf c} _0^{\frac{1}{1-b\theta}} .
\]
\end{lemma}

\begin{proof}
We follow the argument in \cite[p. 1004]{10.1215/S0012-7094-84-05145-7}.

We have
\begin{equation*}
\ln \xi (k) \leq \ln {\bf c} _0 + \sigma \ln \frac{1}{h-k} + \theta \ln \xi (h) .
\end{equation*}

Choosing \(k = t^b\) and \(h = t\), we obtain
\begin{equation}\label{159}
\int ^1 _{(1/2)^{1/b}} \frac{\ln \xi (t^b)}{t} \, \din t \leq  \theta \int ^1 _{(1/2)^{1/b}} \frac{\ln \xi (t)}{t} \, \din t  + C _1,
\end{equation}
where
\begin{equation}\label{161}
\begin{aligned}
C _1 := & \,  \int ^1 _{(1/2)^{1/b}} \frac{\ln  {\bf c} _0 }{t}+  \frac{\sigma}{t} \ln \frac{1}{t-t^b} \, \din t\\
 \leq & \,- \ln  {\bf c} _0\ln 2^{-1/b}  + \frac{\sigma}{2^{-1/b}} \int ^1 _{(1/2)^{1/b}} \ln \frac{1}{t} + \ln \frac{1}{1-t^{b-1}} \, \din t\\
 \leq & \, \left( b^{-1}\ln {\bf c} _0  + 2\sigma \right)\ln 2+  \frac{\sigma}{b-1} \int ^1 _{1/2} -\ln s \, \din s, \qquad s= 1-t^{b-1}.
\end{aligned}
\end{equation}

Hence, by the change of variables \(z = t^b\) in the left-hand side of \eqref{159}, we obtain
\begin{equation*}
\left( \frac{1}{b} - \theta \right) \int ^1 _{1/2} \frac{\ln \xi (s)}{s} \, \din s \leq   C _1.
\end{equation*}


By the assumptions on \(\xi\), we conclude that
\begin{equation*}\label{160}
\left(\ln 2\right) \ln \xi (1/2)\leq \frac{C _1}{1/b - \theta}.
\end{equation*}

Taking \eqref{161} into account, we obtain
\begin{align*}
\xi (1/2) \leq & \, \exp \left( \frac{(\ln 2)b^{-1} \ln {\bf c} _0 + \sigma (b-1)^{-1} C _2}{(\ln2)(1/b -\theta)}\right)\\
= & \,C_3^{\frac{\sigma b}{(1-b\theta)(b-1)}} {\bf c} _0 ^{\frac{1}{1-b\theta}} .
\end{align*}
This concludes the proof of the lemma.
\end{proof}

We  assume that \(\rdo \in (0,\tfrac{1}{2})\) is sufficiently small so that
\[
\overline{\{x\in\Omega \mid \piu(x)<\rdo\}} \cap \Sigma = \emptyset.
\]

For \(\ell\in(0,\rdo)\) and \(r\in(0,\ell)\), we define
\[
\vltr_{\ell,r}:=\left\{x\in\Omega \mid |\piu(x)-\ell|<r\right\}.
\]



\subsection{Proof of Proposition \ref{266}}

The proof follows from the following lemma and Corollary \ref{277}.
\begin{lemma} \label{219}
Under the assumptions \ref{fnf16}, \ref{fnf21}, \ref{43}, \eqref{224}, \eqref{225}, \ref{102}--\ref{106}. Suppose that \(u\in  W^{1, p(x)  } _{\operatorname{loc}} (\overline{\Omega} \setminus (\Gamma \cup \Sigma)  ; \wghtv)\)
\(\cap L^{\infty} _{\operatorname{loc}} (\overline{\Omega} \setminus (\Gamma \cup \Sigma) )\) satisfies
\begin{equation} \label{222}
\int _{\Omega }\left(  \mathscr{A}_p ( \cdot , \grad u) \bullet \deriv \zeta  +  f (\cdot , u) \zeta  \right) \, \din \wghtv  + \int _{\partial \Omega} b( \cdot , u)   \zeta \, \ds \leq 0,
\end{equation}
for all \(\zeta \in W_{\loc }^{1, p (x)}(\overline{\Omega} \setminus (\Gamma \cup \Sigma) ; \wghtv) \cap L_{\operatorname{loc}}^{\infty}(\overline{\Omega} \setminus (\Gamma \cup \Sigma) )\) with \(\zeta \geq 0\),  and \(\operatorname{supp} \zeta  \subset \overline{\Omega} \setminus (\Gamma \cup \Sigma)\).

Then, if \(\ell<2r\) and \(0<r<\ell<\rdo\), we have the estimate
\begin{equation} \label{242}
\max \{u ,0\}   \leq C_0 r^{-\tau}, \quad \text { \textit{a.e.} in } \quad  \vltr _{\ell, r/2},
\end{equation}
where
\begin{align*}
\tau =  \displaystyle \frac{1}{C_1}\left[  p^+  - \frac{\epsilon n_d(\gaglia -p^{-})}{\gaglia }  \right] \geq \frac{p^+}{q - p^+ +1},
\end{align*}
with $C_1 = (  q      - p^+   +1)-(  q      - p^-   +1)\epsilon $ and $\epsilon \in \left( 0, \tfrac{q - p^+ +1}{q - p^- +1} \right)$. Also,
\[
\begin{aligned}
C_0={\bf c_1} ^{\frac{8(\beta +q) n}{\min \{C_1 , 1\} ^2}  } \left[ C _2  (\consgaglia +1)^{n}\frac{\gaglia+p^{-} (\gaglia -1) \epsilon + \gaglia (1-\epsilon)q  }{( \gaglia - p^- )\epsilon} \right]^{\frac{2\beta}{C_1} },
\end{aligned}
\]
with \(\beta   =  1 +p^- \frac{\gaglia -  1}{\gaglia   }\epsilon\) and \(C_2 = C_2 (\Lambda _1 , \Lambda _2 , \ctb _1 , n_d , n)>1\), where \({\bf c_1}\) is given in Lemma \ref{64}.
\end{lemma}

\begin{proof} {\bf Step 1.} We make the substitution \(u=c_\ast w\), where \(c_\ast>1\) is a constant
to be chosen below; see \eqref{137}. We assume that
\[
 \wghtv\left(  \{x \in \mathcal V_{\ell,r/2}\mid w(x)>0\}\right)\neq 0,
\]
since otherwise \eqref{242} is immediate. Set
$$
\Omega_+ := \left\{x \in \mathcal V_{\ell,r} \;\middle|\; w(x)>0\right\}.
$$

Let
$$
m_t:=\operatorname{ess} \operatorname{sup} \left\{ w(x)\;\middle|\; x \in \vltr_{\ell,tr} \cap \Omega_+ \right\}, 
\qquad \frac12 \le t \le 1.
$$

Let also $\xi:\mathbb{R}\to[0,1]$ be a smooth function satisfying
\begin{equation*}
\xi(a)=
\begin{cases}
0 & \text{if } a\in (-\infty, sr],\\
1 & \text{if } a \in \left[\tfrac{(s+t)r}{2}, \infty\right),
\end{cases}
\end{equation*}
with \(\xi>0\) on \( (sr,\tfrac{(s+t)r}{2} )\), and
$$
|\xi'| \le \frac{c_1}{r(t-s)} \qquad \text{on } \mathbb{R}.
$$

Define $\widetilde \xi : \Omega \to \mathbb{R}$ by
\[
\widetilde \xi(x)=\xi\bigl(|\piu(x)-\ell|\bigr).
\]
Then, for some $c_1>1$,
\begin{equation}\label{163}
\bigl| \grad \widetilde \xi \bigr|_{\fnf}=\bigl| \deriv \widetilde \xi \bigr|_{\fnfs}
\le \frac{c_1}{r(t-s)} \qquad \text{on } \Omega .
\end{equation}

Let $\tfrac12 \le s<t \le 1$. Define the functions $z : \Omega\to\mathbb{R}$ and $z_k : \Omega\to\mathbb{R}$ by
\begin{align*}
z(x)&:= w(x) - m_t \,\xi\left(\left| \piu(x)-\ell \right|\right), \\[5pt]
z_k(x)&:=
\begin{cases}
\max \left\{ w(x)-\Bigl[m_t\,\xi\bigl(\left| \piu(x)-\ell \right|\bigr)+k\Bigr],\,0 \right\}
& \text{if } x\in \vltr_{\ell,tr},\\
0 & \text{if } x\in \Omega \setminus \vltr_{\ell,tr},
\end{cases}
\end{align*}
where $0 \le k\le \operatorname{ess} \operatorname{sup}_{\Omega_+} z$.

By \ref{102}, we have
\begin{gather}
\spt\, z_k \subset \left\{ x\in \vltr_{\ell,tr} \;\Big|\; \left(\frac{\ctb_0 m_t}{k}\right)^{\frac{1}{\boldsymbol{\delta}}} \ge \pid(x) \right\}
\quad \text{if } k>0,\label{300} \\[5pt]
z_0 \le w \le \ctb_0 m_t \left(1+\pid(x)\right)^{-\boldsymbol{\delta}}.\label{301}
\end{gather}


\medskip

{\bf Step 2.} We assume that \( m_{1/2} > 1 \); the conclusion is clearly valid when \( m_{1/2} \leq 1 \).

Let \( k \in (0, \mathcal{K}_0)\), where
$$
\mathcal{K}_0 := \sup \Bigl\{k\in [0 , \operatorname{ess} \operatorname{sup} _{\Omega^\prime} z] \ \Big| \  \wghtv\left( \{ x\in \vltr _{\ell , tr} \mid z_k (x) >0\} \right) \neq 0\Bigr\}.
$$
Observe that \( \mathcal{K}_0 \geq m_s \geq m_{1/2} >1 \).

Substituting \( \zeta = z_k \) into \eqref{222}, we obtain
\begin{equation*} 
\int _{\Omega }   \mathscr{A}_p ( \cdot , \grad u) \bullet \deriv z_k   +  f(\cdot , u)  z_k  \, \din \wghtv + \int _{\partial \Omega} b(\cdot , u)     z_k \, \ds \leq 0. 
\end{equation*}

Then,
\begin{equation*} 
 \int _{\Omega _k}    \mathscr{A}_p ( \cdot , \grad u) \bullet  \left( \deriv w - m_t \deriv \widetilde \xi  \right)    \, \din \wghtv +  \int _{\Omega _k}    f( \cdot , u ) z_k \, \din \wghtv + \int _{\partial \Omega} b(\cdot , u)     z_k \, \ds \leq 0, 
\end{equation*}
where 
$$
\Omega _k  :=\left\{x \in \vltr _{\ell , tr}\mid z_k (x)>0\right\}.
$$

Using \ref{43}, we obtain
\begin{equation}\label{237}
\begin{aligned} 
&\Lambda _1 \int_{\Omega  _k }    c_{\ast} ^{p - 1}|\grad w|^{p} _{\fnf }    \, \din \wghtv    
+ \int_{\Omega  _k}    c _\ast^{q   } k^{q   } z_k   \, \din \wghtv \\
&\qquad \leq \Lambda _2 c _1  \int_{\Omega  _k } \frac{m_t c_\ast ^{p-1}}{r(t-s)}  | \grad w|_{\fnf }^{p-1}  \, \din \wghtv,
\end{aligned}
\end{equation}
by \eqref{162}. 

Using Young's inequality, we obtain
\begin{equation} \label{238}
\begin{aligned}
   \int_{\Omega  _k }  \frac{  m _t  c_\ast ^{p-1}}{r(t-s)}   |\grad w|_{\fnf }^{p-1} \, \din \wghtv
   \leq & \, \displaystyle  \int_{\Omega  _k } c _\ast ^{p-1} \left\{\frac{1}{p} \epsilon _1 ^{-(p-1)}\left[\frac{m _t}{r(t-s)}\right]^{p} + \frac{p-1}{p} \epsilon _1 |\grad w|^p _{\fnf}\right\} \, \din \wghtv\\[5pt]
 \leq & \, \displaystyle    c_\ast ^{p^+-1} c _2 \left[\frac{m _t}{r(t-s)}\right]^{p^+} \wghtv (\Omega  _k ) + \frac{p^+-1}{p^+}\epsilon _1 \int _{\Omega _k}c_\ast ^{p-1} |\grad w|^{p}  _{\fnf} \, \din \wghtv ,
\end{aligned}
\end{equation}
where \( c _2 =\tfrac{1}{p^-} \max \{\epsilon _1 ^{-p^+ +1} , \epsilon _1 ^{-p^- +1} \} \). In the last inequality, we used \(m_t >1\) for \(1/2\leq t\leq 1\), and \(c _\ast >1\).

Choose $
\epsilon _1 = \tfrac{\Lambda _1 p^+}{2\Lambda _2 c _1(p^+-1)}$. Then, by \eqref{237} and \eqref{238},
\begin{align*}
\frac{\Lambda _1}{2}\int_{\Omega _k   }     c_{\ast} ^{p- 1}  |\grad w|_{\fnf} ^{p}    \, \din \wghtv  + \int_{\Omega  _k } c _\ast ^{q   } k^{q   }z_k  \, \din \wghtv  
 \leq   c_3 c_\ast ^{p ^+  -1 } \left[\frac{m_t}{r (t-s)} \right]^{p ^+ }  \wghtv (\Omega _k),
\end{align*}
where 
$$
c _3 = \frac{c_1 \Lambda _2}{p^{-}} \max \left\{\left[ \frac{\Lambda_1 p^{+}}{2 \Lambda_2 c _1\left(p^{+}-1\right)} \right]^{-p^{+}+1}, \left[ \frac{\Lambda_1 p^{+}}{2 \Lambda_2 c _1\left(p^{+}-1\right)} \right]^{-p^{-}+1}\right\}.
$$

Note that 
$$
\left| \grad z_{k} \right|_{\fnf} \leq \left| \grad w \right|_{\fnf} + m_{t} \bigl|\grad \widetilde \xi\bigr|_{\fnf}  \quad \text{in } \Omega  _k .
$$
Therefore,
\begin{equation*}
\begin{aligned}
& \frac{\Lambda _1}{2}\int _{\Omega  _k } c_{\ast} ^{p - 1}    \left( \frac{1}{2^{p^+ -1}} \left|\grad z_{k}\right|_{\fnf}^{p  } - m_t ^p \bigl|\grad \widetilde \xi \bigr|_{\fnf}^p \right) \, \din \wghtv   
+ \int_{\Omega  _k } c_\ast ^{q   } k^{q } z_k  \, \din \wghtv 
 \leq c _3 c_\ast ^{p^+ -1 } \left[\frac{m_t}{r (t-s)} \right]^{p ^+ } \wghtv (\Omega _k ) .
\end{aligned}
\end{equation*}

Employing \( |\grad z_k|_{\fnf} ^{p^-} \leq 1 + |\grad z_k|_{\fnf} ^p \) in \( \Omega _k \), together with \eqref{163}, we obtain
\begin{equation}\label{112}
\begin{aligned}
c_{\ast} ^{p ^- - 1} \int_{\Omega _k   }       \left|\grad z_k\right|_{\fnf}^{p ^-}    \, \din \wghtv  + c _\ast ^{q }  k^{q } \int_{\Omega  _k }  z_k  \, \din \wghtv  
 \leq  & \, \displaystyle  c_\ast ^{p ^+  -1 } \left\{ c_4   \left[\frac{m_t}{r (t-s)} \right]^{p ^+ } +1 \right\} \wghtv (\Omega _k )\\[3pt]
 \leq  & \, \displaystyle (c _4 +1) c_\ast ^{p ^+  -1 } \left[\frac{m_t}{r (t-s)} \right]^{p ^+ }  \wghtv (\Omega _k),
 \end{aligned}
\end{equation}
where 
$$
c _4 := 2^{p^+-1}\left(\max \left\{\frac{2}{\Lambda _1} , 1\right\}c _3 + c _1 ^{p^+}\right) \leq 2^{2p^+-1}  \Lambda_2 \max \left\{\frac{2}{\Lambda _1} , 1\right\}  \left(\frac{ \Lambda _2}{\Lambda _1} \right)^{p^+-1}   c _1^{p^+}.
$$


\medskip

{\bf Step 3.} Using condition \ref{106}, 
\begin{equation}\label{239} 
\left(\int_{\Omega  _k}z_{k}^{\gaglia   } \, \din \wghtv \right)^{\frac{1}{\gaglia   }}  \leq  \consgaglia \left(\int _{\Omega  _k}\left|\grad z_{k}\right|_{\fnf}^{p^-} \, \din \wghtv   \right)^{\frac{1}{p^- }} .
\end{equation}

Using Hölder's inequality, we get
\begin{equation}\label{240}
\int_{\Omega _k} z_{k} \, \din \wghtv  \leq \left(\int_{\Omega  _k} z_{k}^{ \gaglia   } \, \din \wghtv \right)^{\frac{1}{\gaglia   }} \wghtv(\Omega  _k )^{\frac{\gaglia    -1}{ \gaglia    }} .
\end{equation} 

From \eqref{112}--\eqref{240}, we have
 \begin{equation*} 
\begin{aligned}
& \epsilon c_\ast  ^{p^-   - 1}  \displaystyle \left[ \consgaglia ^{-1} \wghtv ( \Omega  _k ) ^{-\frac{\gaglia    -  1}{\gaglia   }}   \int_{\Omega _k} z_k \, \din \wghtv  \right] ^{p ^-} 
  \\[3pt]
 &  + \displaystyle (1-\epsilon) c _\ast ^{q }   k^q   \int _{\Omega _k} z_k \, \din \wghtv   
  \leq     \displaystyle  \left( c _4 +1\right)c _\ast ^{p ^+     -1 } \left[\frac{m_t}{r (t-s)} \right]^{p  ^+ } \wghtv(\Omega _k )  .
\end{aligned}
\end{equation*}
 

Then,
\begin{equation} \label{215}
\begin{aligned}
 & \epsilon c_\ast  ^{p^-   - 1} \displaystyle  \left[ \consgaglia ^{-1} \int_{\Omega _k} z_k \, \din \wghtv  \right] ^{p ^-} 
    + \displaystyle (1-\epsilon) c _\ast ^{q }   k^q  \wghtv(\Omega _k ) ^{ p^-\frac{\gaglia -1}{\gaglia}}  \int _{\Omega _k} z_k \, \din \wghtv  \\[5pt]
 & \leq    \displaystyle  (  c _4 +1) c _\ast ^{p ^+     -1 } \left[\frac{m_t}{r (t-s)} \right]^{p  ^+ } \wghtv (\Omega _k ) ^{1+ p^-\frac{\gaglia -1}{\gaglia}}.
\end{aligned}
\end{equation}

Applying Young's inequality $|a|^{\epsilon} |b|^{1-\epsilon}\leq \epsilon |a| + (1-\epsilon)|b|$ to the left-hand side of \eqref{215}, we obtain
\begin{equation*}
\begin{aligned}
&  k^{(1-\epsilon)q}    \displaystyle c _\ast ^{(p^- -1 )\epsilon + q(1-\epsilon)}    \wghtv (\Omega _k ) ^{ p^-\frac{\gaglia -1}{\gaglia} (1-\epsilon)}  \left(\int _{\Omega _k} z_k \, \din \wghtv \right)^{p^- \epsilon + 1 - \epsilon}\\[5pt]
& \leq   \displaystyle (c _4 +1) \consgaglia ^{\epsilon p^-} c _\ast ^{p ^+     -1 } \left[\frac{m_t}{r (t-s)} \right]^{p  ^+ } \wghtv ( \Omega _k ) ^{1+ p^-\frac{\gaglia -1}{\gaglia}}.
\end{aligned}
\end{equation*}
   
Therefore,
\begin{equation} \label{217}
\begin{aligned}
   k^{\frac{(1-\epsilon)q }{\beta}}  & \displaystyle  c_\ast ^{\frac{(  q      - p^+   +1)-(  q      - p^-   +1)\epsilon}{\beta  }} \\
\leq & \, \displaystyle  c_{5} \left(   \int_{\Omega _k } z_k \, \din \wghtv  \right)^{-\frac{ 1 + (p^- - 1)\epsilon }{\beta  }} \left[\frac{m_t}{r (t-s)} \right]^{\frac{p  ^+ }{\beta  }}   \wghtv(\Omega _k) ,
\end{aligned}
\end{equation}
where $c _5 = (c _4 +1)^{1/\beta}\consgaglia ^{\epsilon p^-/\beta}$, $\beta   =  1 +p^- \frac{\gaglia -  1}{\gaglia   }\epsilon $, and $k>0$.

\medskip

{\bf Step 4.} Integrating \eqref{217} with respect to $k$,
\begin{equation*}
\begin{aligned}
  c_\ast ^{\frac{(  q      - p^+   +1)-(  q      - p^-   +1)\epsilon}{\beta  }}     & \displaystyle   \int _0 ^{\mathcal{K} _0} k^{\frac{(1-\epsilon)q }{\beta}}   \, \din k \\
\leq & \, \displaystyle  c_{5}  \left[\frac{m_t}{r (t-s)} \right]^{\frac{p  ^+ }{\beta  }} \int _0 ^{\mathcal{K} _0} \left(   \int_{\Omega _k } z_k \, \din \wghtv  \right)^{-\frac{ 1 + (p^- - 1)\epsilon }{\beta  }}  \wghtv( \Omega _k) \, \din k .
\end{aligned}
\end{equation*}



Let us consider the equality
$$
\frac{\textnormal d}{\textnormal d k}  \left(   \int_{\Omega _k } z_k \, \din \wghtv  \right) = - \wghtv (\Omega _k)
$$


Since $1- \frac{1+(p^- -1)\epsilon}{\beta} >0$ and $\mathcal{K} _0 >1$,
\begin{equation*}
\begin{aligned}
  c_\ast ^{\frac{(  q      - p^+   +1)-(  q      - p^-   +1)\epsilon}{\beta  }}     & \displaystyle   \mathcal{K} _0^{1+ \frac{(1-\epsilon)q}{\beta} }   \\
 \leq & \, \displaystyle  c_{6}  \left[\frac{m_t}{r (t-s)} \right]^{\frac{p  ^+ }{\beta  }}  \left(   \int_{\Omega _0 } z_0 \, \din \wghtv  \right)^{1-\frac{ 1 + (p^- - 1)\epsilon }{\beta  }}  ,
\end{aligned}
\end{equation*}
where 
$$
c _6 := c _5 \frac{\beta +(1-\epsilon)q  }{\beta -1 -(p^- -1)\epsilon} = c _5 \frac{\gaglia+p^{-} (\gaglia -1) \epsilon + \gaglia (1-\epsilon)q  }{( \gaglia - p^- )\epsilon} .
$$

We note that $\mathcal{K} _0 \geq m_s$. Also, by applying \ref{104} and \eqref{301}, we obtain
\begin{equation*}
\begin{aligned}
  m_s ^{1+ \frac{(1-\epsilon)q}{\beta} } &  \displaystyle  c_\ast ^{\frac{(  q      - p^+   +1)-(  q      - p^-   +1)\epsilon}{\beta  }}     \displaystyle     \\
 \leq  & \, \displaystyle 3^{n_d}  \ctb _1 c_{6}  \left[\frac{m_t}{r (t-s)} \right]^{\frac{p  ^+ }{\beta  }}  m_t ^{1-\frac{ 1 + (p^- - 1)\epsilon }{\beta  }} r^{n_d\left[1-\frac{ 1 + (p^- - 1)\epsilon }{\beta  } \right]} ,
\end{aligned}
\end{equation*}
because $\ell +tr\leq 3r$.

Now we take $c _\ast >0$ such that
$$
c_\ast ^{\frac{(  q      - p^+   +1)-(  q      - p^-   +1)\epsilon}{\beta  }}   =
r^{- \frac{p^+}{\beta}+ n_d\left[1-\frac{ 1 + (p^- - 1)\epsilon }{\beta  } \right]}.
$$


Hence, 
\begin{equation}\label{137}
c _\ast  = r^{-\tau},
\end{equation} 
where
\begin{align*}
\tau = & \, \displaystyle -\left\{   - \frac{p^+}{\beta} + n_d\left[1-\frac{ 1 + (p^- - 1)\epsilon }{\beta  } \right] \right\}\\[3pt]
& \, \displaystyle \cdot \left[\frac{(  q      - p^+   +1)-(  q      - p^-   +1)\epsilon}{\beta  } \right]^{-1} \geq \frac{p^+}{q - p^+ +1},
\end{align*}
if $\epsilon  \in (0, \frac{ q      - p^+   +1}{ q      - p^-   +1})$. 

On the other hand,
$$
m_s \leq (3^{n_d} \ctb _1 c _{6})^{\frac{\beta}{p^+} \sigma} \frac{m_t ^{\theta}}{(t-s)^{\sigma}},
$$
where
\begin{align*}
\theta = & \, \displaystyle  \left[ \frac{p^{+}}{\beta} +  1-\frac{1+(p^{-}-1) \epsilon}{\beta} \right]
 \left[ 1+\frac{(1-\epsilon) q}{\beta} \right] ^{-1}<1\\
\sigma = & \, \displaystyle \frac{p^+}{\beta}\left[1+\frac{(1-\epsilon) q}{\beta}  \right] ^{-1}.
\end{align*}

By virtue of Lemma \ref{64}, we obtain 
$$
m_{1 /2} \leq c _7 := {\bf c}_1 ^{\frac{\sigma b}{(1-b\theta)(b-1)}}(3^{n_d} \ctb _1 c _6)^{\frac{\sigma\beta}{p^+(1-b\theta)} },
$$
where $b=\frac{1+\min \{\theta ^{-1}  ,3\}}{2}$.  Moreover,
$$
\begin{aligned}
c _7 \leq & \,  {\bf c}_1 ^{\frac{\sigma b}{(1-b\theta)(b-1)}} \Biggl[ 3^{n_d} \ctb _1 \left(2^{2p^+}\Lambda_2\max \left\{\frac{2}{\Lambda _1} , 1\right\} \left(\frac{ \Lambda _2}{\Lambda _1} \right)^{p^+-1}  c _1^{p^+}  \consgaglia ^{\epsilon p^-} \right)^{1/\beta}\Biggr.\\[3pt]
& \,  \cdot \Biggl. \frac{\gaglia+p^{-} (\gaglia -1) \epsilon + \gaglia (1-\epsilon)q  }{( \gaglia - p^- )\epsilon} \Biggr]^{\frac{\sigma\beta}{p^+(1-b\theta)} }\\[5pt]
\leq & \, {\bf c_1} ^{\frac{8(\beta+q)  n}{\min \{c _8 , 1\} ^2}  } \left[ c _9 (\Lambda _1 , \Lambda _2 , \ctb _1 , n_d , n)  (\consgaglia +1)^{n}\frac{\gaglia+p^{-} (\gaglia -1) \epsilon + \gaglia (1-\epsilon)q  }{( \gaglia - p^- )\epsilon} \right]^{\frac{2\beta}{c _8} },
\end{aligned}
$$
where $\beta   =  1 +p^- \frac{\gaglia -  1}{\gaglia   }\epsilon $, $c _8 := q-p^{+}+1-\epsilon(q-p^{-}+1)$. Furthermore,
\begin{gather*}
1-b\theta  = \frac{2 - \theta - \theta \min \{\theta ^{-1} , 3\} }{2}\geq \frac{1-\theta}{2} \geq \frac{ q -p^++1 -\epsilon (q -p^- +1)}{2(\beta +q)},\\[3pt]
b-1 = \frac{\min \left\{\theta^{-1}, 3\right\} -1 }{2}\geq \min \left\{\frac{1-\theta}{2\theta} , 1\right\} \geq \min \left\{\frac{ q -p^++1 -\epsilon (q -p^- +1)}{2(\beta +q)} ,1 \right\}.
\end{gather*}


From the substitution $u=c _\ast w$, we obtain
$$
\operatorname{ess} \operatorname{sup} \left\{u(x) \mid x \in \mathcal V_{\ell , r/2} \cap \Omega _+ \right\}=c _\ast m_{1/2} \leq c_{7} c _\ast =  c _{7} r^{-\tau}.
$$
Therefore, we conclude the proof of the proposition. \end{proof}

\begin{corollary} \label{277}
Under assumptions \ref{135}, \ref{136}, \ref{43}, \eqref{224}, \eqref{225}, and \ref{102}--\ref{106}, suppose that $v\in  W^{1, p(x)  } _{\operatorname{loc}} (\overline{\Omega} \setminus (\Gamma \cup \Sigma)  ; \wghtv )
\cap L^{\infty} _{\operatorname{loc}} (\overline{\Omega} \setminus (\Gamma \cup \Sigma) )$ satisfies
\begin{equation*}
\int _{\Omega }\left(  \tilde{\mathscr{A}}_p ( \cdot , \grad (-v)) \bullet \deriv \zeta  -  f (\cdot , -v) \zeta  \right) \, \din \wghtv  - \int _{\partial \Omega}   b( \cdot ,- v)   \zeta \, \ds \leq 0,
\end{equation*}
for all $\zeta \in W_{\loc }^{1, p (x)}(\overline{\Omega} \setminus (\Gamma \cup \Sigma) ; \wghtv) \cap L_{\operatorname{loc}}^{\infty}(\overline{\Omega} \setminus (\Gamma \cup \Sigma) )$ with $\zeta \geq 0$   and $\operatorname{supp} \zeta  \subset \overline{\Omega} \setminus (\Gamma \cup \Sigma)$. 

Then, if \(\ell<2r\) and \(0<r<\ell<\rdo\), we have the estimate
\begin{equation*}
\max \{v ,0\}   \leq \tilde C_0 r^{-\tau}, \quad \text{\textit{a.e.} in } \quad  \vltr _{\ell, r/2},
\end{equation*}
where
\begin{align*}
\tau =  \displaystyle \frac{1}{ \tilde C_1}\left[  p^+  - \frac{\epsilon n_d(\gaglia -p^{-})}{\gaglia }  \right] \geq \frac{p^+}{q - p^+ +1},
\end{align*}
with $\tilde C_1 = (  q      - p^+   +1)-(  q      - p^-   +1)\epsilon$ and $\epsilon \in ( 0, \tfrac{q      - p^+   +1}{q      - p^-   +1})$. Also,
$$
\begin{aligned}
\tilde C_0={\bf c_1} ^{\frac{8(\beta+q)  n}{\min \{ \tilde C_1 , 1\} ^2}  } \left[ \tilde C _2  (\consgaglia +1)^{n}\frac{\gaglia+p^{-} (\gaglia -1) \epsilon + \gaglia (1-\epsilon)q  }{( \gaglia - p^- )\epsilon} \right]^{\frac{2\beta}{\tilde C_1} },
\end{aligned}
$$
with $\beta   =  1 +p^- \frac{\gaglia -  1}{\gaglia   }\epsilon $ and $\tilde C_2 = \tilde C_2 (\Lambda _1 , \Lambda _2 , \ctb _1  , \revc _{\fnf} , n_d , n)>1$, where \({\bf c_1}\) is given in Lemma \ref{64}.
\end{corollary}

Proceeding in the same way as in Lemma \ref{219}, we obtain the following result.

\begin{corollary}[Without the boundary condition] \label{269}
Under assumptions \ref{fnf16}, \ref{fnf21}, \ref{43}, \eqref{224}, \eqref{225}, and \ref{102}--\ref{106}, suppose that $u \in W_{\operatorname{loc}  }^{1  , p (x)  }( \Omega  \setminus \Gamma ; \wghtv) \cap L_{\operatorname{loc}  }^{\infty}( \Omega  \setminus \Gamma)$ satisfies
\begin{equation} \label{272}
\int_{\Omega} \left( \mathscr{A}_p(\cdot , \grad u) \bullet \deriv \zeta  +  f(  \cdot , u ) \zeta  + \widehat  b (\cdot , u) \zeta    \right)  \, \din \wghtv \leq 0,
\end{equation}
for all $\zeta \in W_{\operatorname{loc}  }^{1, p(x)  }( \Omega  \setminus \Gamma ; \wghtv ) \cap L_{\operatorname{loc}  }^{\infty}( \Omega  \setminus \Gamma)$ with $\zeta \geq 0$ and $\operatorname{supp}\zeta \subset  \Omega  \setminus \Gamma$.  Here $\widehat b:\Omega\times \mathbb{R}\to \mathbb{R}$ is measurable and locally bounded, and satisfies $\widehat b(x,u)\operatorname{sign}u\geq 0$ for a.e. $x\in \Omega$ and all $u\in \mathbb{R}$.

Then, if $\ell < 2r$ and $0<r<\ell<\rdo$, we have the estimate
$$
\max \{u , 0\} \leq {\bf C} r^{-\tau} \quad\text{\textit{a.e.} in}\quad \vltr _{\ell, r / 2},
$$
where the constants ${\bf C}$ and $\tau$ are as in Lemma \ref{219}.
\end{corollary}

\section{Boundedness of Solutions}\label{158}



\begin{proposition} \label{267}
Under assumptions \ref{fnf16}, \ref{fnf21}, \ref{43}, \eqref{224}--\eqref{243}, and \ref{102}--\ref{106}, suppose that $u \in W_{\operatorname{loc}}^{1, p (x)}(\overline{\Omega} \setminus ( \Gamma \cup \Sigma ); \wghtv)\cap L_{\operatorname{loc}}^{\infty}(\overline{\Omega} \setminus (\Gamma \cup \Sigma))$ satisfies \eqref{222}. Then
$$
\left(1+  \pid (x) \right)^{\boldsymbol{\delta}}\max \{u (x), 0\} \in L^{\infty}_{\loc}( \overline \Omega \setminus \Sigma) .
$$
\end{proposition}

\begin{proof}
{\bf Step 1.} First, we prove that
\begin{equation} \label{46}
\max\{u,0\} \in L^{\infty}_{\loc}(\overline{\Omega}\setminus \Sigma).
\end{equation}

For \(r\in(0,\mathcal R_1)\), where $0<\mathcal R_1<\min\{\mathrm e^{-2},\rdo^2\}$,  we define
\[
\Lambda(r)
=
\mathop{\mathrm{ess\,sup}}
\left\{
\max\{u(x),0\}
\mid 
r\le \piu(x)\le \mathcal R_1,\ x\in\Omega
\right\}.
\]
Assuming that
\[
\lim_{r\to 0^+}\Lambda(r)=\infty,
\]
we shall derive a contradiction. Choose \(\sigma\in(0,\mathcal R_1)\) such that
\(\Lambda(\sigma)>1\).


For sufficiently small \(r\in(0,\sigma^2)\), define the function
\(\psi_r:\mathbb R\to\mathbb R\) by
\[
\psi_r(t)=
\left\{
\begin{aligned}
&0
&& \text{if } t<r, \\[3pt]
&\frac{2}{\ln(1/r)}\ln\frac{t}{r}
&& \text{if } r\le t\le \sqrt r, \\[6pt]
&1
&& \text{if } t>\sqrt r.
\end{aligned}
\right.
\]

Set
\[
\zeta(x)=
\left\{
\begin{aligned}
&
\left(
\ln \max\left\{\frac{u(x)}{\Lambda(\sigma)},1\right\}
\right)
\psi_r^\alpha\bigl(\piu(x)\bigr)
&& \text{if } \piu(x)\le \sigma, \\[8pt]
&0
&& \text{if } \piu(x)>\sigma,
\end{aligned}
\right.
\]
where
\[
\alpha
:=
\mathop{\mathrm{ess\,sup}}_{x\in\Omega}
\frac{p(x)q}{q-p(x)+1}.
\]

We have $\zeta \in W_{\operatorname{loc}}^{1, p(x)}(\overline{\Omega} \setminus ( \Gamma \cup \Sigma ); \wghtv) \cap L_{\operatorname{loc}}^{\infty}(\overline{\Omega} \setminus (\Gamma \cup \Sigma))$, and, by \ref{102},
\begin{equation}\label{114}
\begin{aligned}
\operatorname{supp}\zeta \subset & \, \displaystyle \left\{x \in \overline{\Omega}  \mid \sigma \geq  \piu(x)   \geq r\right\} \cap \left\{x \in \overline{\Omega}   \mid \ctb _0^ \frac{1}{\boldsymbol{\delta}} \geq  \pid(x)  \right\}\\
\subset & \, \displaystyle \overline \Omega  \setminus (\Gamma \cup \Sigma).
\end{aligned}
\end{equation}
 
Observe that, by \ref{104} and \eqref{114},
\begin{align}
  \int_{\spt \zeta \cap \{r \leq \piu     \leq \sqrt{r}\}}    & \displaystyle    \piu      ^{-\alpha} \, \din \wghtv \leq \frac{\ctb _2 \ctb _0 ^{d/\boldsymbol{\delta}}}{n_d-\alpha  } r^{\frac{n_d-\alpha }{2}} .  \label{115}
\end{align}

\medskip

 {\bf Step 2.} Define $\widetilde \psi  _r : \Omega \to \mathbb{R}$ by $$
 \widetilde \psi  _r (x) = \psi _r \circ \piu(x)   .
 $$
  Then
\begin{equation}\label{113}
\left|\grad \widetilde \psi  _r  \right|_{\fnf} =\left|\deriv \widetilde \psi  _r  \right|_{\fnfs}   \leq    \frac{2  }{  \piu    \ln \frac{1}{r}}\quad \text { in }   \Omega \cap \{r< \piu <\sqrt{r}\}.
\end{equation}

 Substituting $\zeta$ into \eqref{222}, we obtain
$$
\begin{aligned}
 \int_{\Omega_\sigma} \left[ \widetilde \psi _r^\alpha \left( \mathscr{A}_p(  \cdot , \grad u) \bullet u^{-1}\deriv u  \right)
+   \alpha \widetilde \psi _r^{\alpha-1}\left(\ln \frac{u}{\Lambda(\sigma)}\right)  \mathscr{A}_p ( \cdot ,  \grad u) \bullet \deriv \widetilde  \psi_r   \right] & \, \din \wghtv\\[3pt]
 +\int _{\Omega _ \sigma } f(\cdot , u)  \widetilde \psi_r^\alpha\left(\ln \frac{ u}{\Lambda(\sigma)}\right) \, \din \wghtv
 + \int_{\partial \Omega _\sigma  } & b(\cdot , u)   \widetilde \psi_r^\alpha\left(\ln \frac{u}{\Lambda(\sigma)}\right) \, \ds \leq 0 ,
\end{aligned}
$$
where
$$
 \Omega  _\sigma:=\left\{x \in \Omega   \mid    u(x )>\Lambda(\sigma)\right\}.
$$

By virtue of \ref{fnf16}--\ref{fnf21}, \eqref{162}, \ref{43}, and the fact that $u>\Lambda(\sigma)$ in $\Omega  _\sigma$,
$$
\begin{aligned}
& \int_{\Omega  _\sigma} \Lambda _1 \frac{\widetilde \psi _r^\alpha}{u}|\grad u| _{\fnf}^{p}  \, \din \wghtv  +\int_{\Omega  _\sigma}  \widetilde \psi _r^\alpha\left(\ln \frac{u}{\Lambda(\sigma)}\right) u^{q}  \, \din \wghtv \\[3pt]
 & \leq   \displaystyle \int_{\spt \zeta} \Lambda _2 \alpha \widetilde \psi _r^{\alpha-1} \left|\deriv \widetilde\psi_r \right| _{\fnfs} \left(\ln \frac{u}{\Lambda(\sigma)}\right)|\grad u|_{\fnf}^{p   -1}  \, \din \wghtv  \\[5pt]
 & \leq   \displaystyle \int_{\spt \zeta} \Lambda _2 \alpha \widetilde \psi _r^\alpha u^{-1}\left\{\frac{\epsilon _1 ^{-p+1}}{p}\left[\left|\deriv \tilde\psi_r \right| _{\fnfs}   \widetilde \psi _r^{-1} u\left(\ln \frac{u}{\Lambda(\sigma)}\right)\right]^{p}+\frac{\epsilon_1(p-1)}{p}|\grad u|_{\fnf}^{p}\right\}  \, \din \wghtv,
\end{aligned}
$$
by Young's inequality.

Taking $\epsilon_1= \tfrac{\Lambda _1p^+}{ 2 \alpha \Lambda _2(p^+-1)}$,
\begin{equation} \label{74}
\begin{aligned}
 \int_{\Omega  _\sigma} \frac{\widetilde \psi _r^\alpha}{u}|\grad u|_{\fnf}^{p}  \, \din \wghtv & +\frac{1}{\Lambda _1}\int_{\Omega  _\sigma} \widetilde \psi _r^\alpha\left(\ln \frac{u}{\Lambda(\sigma)}\right) u^{q} \, \din \wghtv \\[3pt]
 \leq  & \, \displaystyle c_1  \int_{\spt \zeta}\left|\deriv \widetilde \psi _r\right|_{\fnfs}^{p} \widetilde \psi _r^{\alpha-p   } u^{p  -1}\left(\ln \frac{ u}{\Lambda(\sigma)}\right)^{p}  \, \din \wghtv ,
\end{aligned}
\end{equation}
where
$$
c _1 :=   2^{p^+}\left(\frac{\Lambda _2}{\Lambda _1}\right)^{p^+  } \max \{\alpha ^{p^- } , \alpha ^{p^+ } \}.
$$

Applying Young's inequality with $\epsilon_2>0$, we obtain
\begin{align*}
& \left| \deriv \widetilde \psi _r\right|_{\fnfs}^{p}  \widetilde \psi _r^{\alpha-p   } u^{p -1}\left(\ln \frac{u}{\Lambda(\sigma)}\right)^{p} \\[3pt]
&  \leq    \displaystyle \left(\ln \frac{u}{\Lambda(\sigma)}\right)
 \left\{\frac{\epsilon _2^{-\frac{p-1}{q-p+1}}(q-p+1)}{q}\left[\left|\deriv \widetilde \psi _r\right|_{\fnfs}^{p}\left(\ln \frac{u}{\Lambda(\sigma)}\right)^{p  -1}\right]^{\frac{q     }{q  - p   +1}}+\frac{\epsilon_2 (p-1)}{q} \widetilde \psi _r^{\frac{(\alpha - p)    q  }{p   -1}} u^{q }\right\}.
 \end{align*}


Taking into account that
$$
\frac{(\alpha-p) q}{p-1} \geq \alpha \quad \text { and } \quad \psi_r^{\prime}(t)=\frac{2}{t \ln \frac{1}{r}}>1 \text { for } t \in(r, \sqrt{r}),
$$
and choosing $\epsilon_2= \tfrac{q}{2 c_1(p^+-1)}$, from \eqref{74} and \eqref{113} we have
\begin{align*}
 \frac{1}{2}\int_{\Omega  _\sigma} & \frac{\widetilde \psi _r^\alpha}{u}  |\grad u|_{\fnf}^{p} \, \din \wghtv  +\frac{1}{\Lambda _1}\int_{\Omega  _\sigma} \widetilde \psi _r^\alpha\left(\ln \frac{u}{\Lambda(\sigma)}\right) u^{q} \, \din \wghtv     \\[3pt]
 \leq & \,  \displaystyle c_2 \int_{\spt \zeta \cap \{r\leq \piu    \leq \sqrt{r}\}}\left(\ln \frac{u}{\Lambda(\sigma)}\right)\left(\ln \frac{u}{\Lambda(\sigma)}\right)^{\frac{\left(p-1\right) q     }{q-p   +1}}\left(  \frac{2}{  \piu     \ln \frac{1}{r} } \right)^{\alpha}\, \din \wghtv \\[5pt]
 \leq & \,  \displaystyle c_2  \left(  2^{-1}\ln \frac{1}{r} \right)^{-\alpha} \int_{\spt \zeta \cap \{r\leq \piu    \leq \sqrt{r}\}}  \left(\ln u \right)^{1+\frac{\left(p-1\right) q     }{q-p   +1}}   \piu     ^{-\alpha} \, \din \wghtv ,
\end{align*}
where
$$
c_2 = \left[\frac{2\Lambda_2(\alpha +1)}{\Lambda_1}\right]^{ \frac{(p^+)^2-1}{q-p^++1}} .
$$

For sufficiently small \(r>0\), we may assume that
$$
2< \ln \frac{{\bf C} ^{1/\tau}}{\piu (x)} < 2 \ln \frac{1}{\piu (x)}  \quad \text { if }\quad 0 <\piu (x) \leq \sqrt{r},
$$
where \(\mathbf C\) and \(\tau\) are the constants introduced in Proposition \ref{266}.

From Proposition \ref{266} and \eqref{115}, we see that
\begin{align*}
 \int_{\Omega  _\sigma} \frac{ \widetilde \psi _r^\alpha}{u}  |\grad u|_{\fnf}^{p}  \, \din \wghtv  +\int_{\Omega  _\sigma} \widetilde \psi _r^\alpha\left(\ln \frac{ u}{\Lambda(\sigma)}\right) u^{q} \, \din \wghtv   
 \leq  & \,  \displaystyle c _3  \int_{\spt \zeta \cap \{r\leq \piu    \leq \sqrt{r}\}}    \left( \ln   \frac{1}{\piu     }  \right)^{1+\frac{(p     -1) q }{q-p   +1}}  
  \piu        ^{-\alpha}  \, \din \wghtv\\[3pt]
 \leq & \, \displaystyle c_3 \left(\ln \frac{1}{r}\right)^{   1+\frac{(p^+     -1) q}{q-p^+   +1} } \int_{\spt \zeta \cap \{r\leq \piu     \leq \sqrt{r}\}}      
  \piu         ^{-\alpha }  \, \din \wghtv\\[5pt]
  \leq & \, \displaystyle \frac{c _3\ctb _2 \ctb _0 ^{d/\boldsymbol{\delta}}}{n_d-\alpha  } \left(\ln \frac{1}{r}\right)^{ 1+\frac{(p^+     -1) q}{q-p^+   +1} }    r^{\frac{n_d-\alpha }{2}}, 
\end{align*}
where
$$
c _3 = c _2 2^\alpha (\tau +1)^{1+\frac{(p^+-1) q}{q-p^++1}}.
$$

Since $\lim _{r\rightarrow 0^+} 1/r = \infty$, we can assume that
$$
\left(\ln \frac{1}{r}\right)^{ 1+\frac{(p^+     -1) q}{q-p^+   +1} } \leq  r^{-\frac{n_d-\alpha }{4}}.
$$
Therefore, as $r\rightarrow 0^+$,
$$
\frac{1}{2}\int_{\Omega  _\sigma} \frac{|\grad u|_{\fnf}^{p }}{u}  \, \din \wghtv + \frac{1}{\Lambda _1}\int_{\Omega  _\sigma}\left(\ln \frac{u}{\Lambda(\sigma)}\right) u^{q}  \, \din \wghtv  =0 .
$$
Hence, $u (x  )=\Lambda(\sigma)$ \textit{a.e.} in $\Omega  _\sigma$. Thus, we obtain a contradiction, which proves \eqref{46}:
$$
\max \{u, 0\} \in L^{\infty} _{\loc}(\overline \Omega \setminus \Sigma).
$$

Finally, by \ref{102}, we conclude the proof of the proposition.
\end{proof}

As in Proposition \ref{267}, we obtain the following result.

\begin{lemma}[Without the boundary condition]
Under the assumptions \ref{fnf16}, \ref{fnf21}, \ref{43}, \eqref{224}-- \eqref{243}, \ref{102}--\ref{106}, suppose that $u \in W_{\operatorname{loc}  }^{1, p (x) }(\Omega   \setminus \Gamma  ; \wghtv )\cap L_{\operatorname{loc}  }^{\infty}(\Omega  \setminus \Gamma)$ satisfies \eqref{272}. Then
$$
\left(1+\pid (x)\right)^{\boldsymbol{\delta}}\max \{u (x), 0\} \in L^{\infty} _{\loc}(  \Omega ) .
$$
\end{lemma}



\section{Removability of the singular set} \label{33}


\subsection{Proof of Theorem \ref{23}} 


{\bf Step 1.} First, we prove that $u \in W^{1, p(x)} _{\loc}(\overline \Omega \setminus  \Sigma ; \wghtv) \cap L^{\infty} _{\loc}(\overline \Omega \setminus  \Sigma)$.

As a consequence of Proposition  \ref{267}, $ (1+\pid)^{\boldsymbol{\delta}}u \in L^{\infty}_{\loc}(\overline \Omega \setminus \Sigma)$. Next, for $r < 2 \rdo / 5$, let $\psi_r: \mathbb{R} \rightarrow \mathbb{R}$ be a smooth function such that
$$
\psi_r(t)=\left\{
\begin{aligned}
&0 & &\text{if } t<\frac{r}{2} \text{ or } t> \frac{5r}{2}, \\
&1 & &\text{if } r<t<2r,
\end{aligned}
\right.
$$
$0 \leq \psi_r \leq 1$ and $\left|\psi_r^{\prime}\right| \leq c / r$, where $c>1$ is a suitable positive constant. Set
$$
\zeta (x)=\left[ (1+\pid (x))^{\boldsymbol{\delta}/2}  \psi_r \circ  \piu(x)   \right] ^{p^+} u(x).
$$
We have $\zeta \in W_{\operatorname{loc}}^{1, p(x)}(\overline{\Omega} \setminus ( \Gamma \cup \Sigma ); \wghtv) \cap L_{\operatorname{loc}}^{\infty}(\overline{\Omega} \setminus (\Gamma \cup \Sigma))$, with $\operatorname{supp}\zeta \subset \overline{\Omega} \setminus (\Gamma \cup \Sigma)$.

Set $\widetilde \psi_r : \Omega \to \mathbb{R}$, defined by $\widetilde \psi_r(x) =(1+\pid(x))^{\boldsymbol{\delta}/2} \psi_r \circ \piu(x)$. Then, by \eqref{278},
\begin{equation}\label{116}
\left|\deriv \widetilde \psi_r(x) \right|_{\fnfs}=\left|\grad \widetilde \psi_r(x) \right|_{\fnf} \leq (\boldsymbol{\delta} /2 + 1)c (1+\pid(x))^{\boldsymbol{\delta}/2} r^{-1}.
\end{equation}

Substituting $\zeta$ into \eqref{101},
$$
\int_{\Omega}   \mathscr{A}_p(\cdot , \grad u)\bullet \left( p ^+ \widetilde \psi _r^{p^+-1}  u \deriv \widetilde \psi  _r  +\widetilde \psi _r^{p ^+} \deriv u \right) 
 +f (\cdot , u) \widetilde \psi _r^{p^+} u  \, \din \wghtv +\int_{\partial \Omega  } b(\cdot , u) \widetilde \psi _r ^{p ^+} u  \, \ds  \leq 0.
$$
By \ref{fnf16}, \ref{fnf21}, and \ref{43}, we have
$$
\begin{aligned}
& \int_{\Omega  } \Lambda _1|\grad u|_{\fnf}^{p} \widetilde \psi _r^{p^+} \, \din \wghtv    
 \leq \int_{\Omega  } \Lambda _2 p^+ \widetilde \psi _r^{p^+  -1}\left|\deriv \widetilde \psi _r\right|_{\fnfs} \cdot |u| \cdot |\grad u|^{p-1} _{\fnf}\, \din \wghtv.
\end{aligned}
$$
Then, by using Young's inequality with $\epsilon_1>0$,
$$
\begin{aligned}
 \int_{\Omega  }|\grad u|^{p } _{\fnf} \widetilde \psi _r^{p^+} \, \din \wghtv 
 \leq & \, \displaystyle \frac{\Lambda _1^{-1} \Lambda _2 p^+ \epsilon _1^{-p+1}}{p^-}\int_{\Omega \cap \{r/2 \leq \piu \leq 5r/2\}  } \left[\left|\deriv \widetilde \psi _r\right| _{\fnfs}  |u|\widetilde \psi  _r ^{p^+ - 1- \frac{p^+(p-1)}{p}}\right]^{p}\, \din \wghtv \\[5pt]
   & \, \displaystyle + \frac{\Lambda _1^{-1} \Lambda _2 p^+ \epsilon _1(p^+-1)}{p^+}\int _{\Omega} \left(|\grad u| _{\fnf}^{p -1} \widetilde \psi _r^{\frac{p^+(p-1)}{p}}\right)^{\frac{p}{p-1}} \, \din \wghtv.
\end{aligned}
$$
Take $\epsilon _1=\tfrac{ \Lambda _1 }{2 (p^+-1) \Lambda _2 }$. By Proposition \ref{267}, \eqref{116}, and \ref{104},
\begin{align}
\int_{\Omega \cap \{r \leq  \piu  \leq 2r\}  }  |\grad u| _{\fnf}^{p } \left( 1+  \pid  \right) ^{\boldsymbol{\delta} p^+/2}  \, \din \wghtv  \leq & \, c_0 r^{-p^+} \int_{\Omega \cap\left\{r / 2 \leq \piu \leq 5 r / 2\right\}} \left[ (1+\pid)^{\boldsymbol{\delta}/2}  |u| \right]^p \, \din \wghtv \label{257} \\[3pt]
\leq & \, \displaystyle c _1 r^{-p^+}\int_{\Omega \cap \{r/2 \leq  \piu \leq 5r/2\}  } \left( 1+  \pid  \right) ^{-\boldsymbol{\delta}/2} \, \din \wghtv   \nonumber \\[5pt]
\leq & \, \displaystyle c_1 \ctb _1 \left(\frac{5}{2} \right)^{n_d} r^{n_d-p^+}, \label{28}
\end{align}
where 
\begin{gather*}
c_0= \frac{p^+}{p^-}\left[ \frac{2 (\boldsymbol{\delta}/2 +1)c \Lambda _2}{\Lambda _1} \right]^{p^+} (n-1)^{p^+ -1},\\
c _1= \frac{p^+}{p^-}\left[ \frac{2 (\boldsymbol{\delta}/2 +1)c \Lambda _2}{\Lambda _1} \left(1+\|u(1+\pid)^{\boldsymbol{\delta}}\|_{L^\infty (\{\piu < \rdo\})}\right)\right]^{p^+} (n-1)^{p^+ -1}.
\end{gather*}

Therefore,
\begin{equation}\label{246}
\begin{aligned}
& \int_{\Omega   \cap\{ \piu  \leq 2 r\}}|\grad u|_{\fnf}^{p}  \left( 1+  \pid  \right) ^{\boldsymbol{\delta} p^+/2}  \, \din \wghtv   \\[5pt]
&\quad  \  =\sum_{i=0}^{\infty} \int_{\Omega   \cap \{2^{-i} r \leq  \piu  \leq 2^{-i+1} r\}}|\grad u| _{\fnf}^{p} \left( 1+  \pid  \right) ^{\boldsymbol{\delta} p^+/2} \, \din \wghtv   \leq c_1 \ctb _1 \left(\frac{5}{2} \right)^{n_d}  \sum_{i=0}^{\infty}\left(\frac{r}{2^i}\right)^{n_d-p ^+} < \infty .
\end{aligned}
\end{equation}
So $|\grad u|_{\fnf} \in L^{p(x)} _{\loc} (\overline \Omega \setminus \Sigma ; \wghtv)$, and thus we have proved that $u \in W^{1, p(x)} _{\loc}(\overline \Omega \setminus \Sigma ; \wghtv) \cap L^{\infty} _{\loc} (\overline \Omega \setminus \Sigma)$.

\medskip

{\bf Step 2.} Now we show that $u$ is a solution of \eqref{1} in the domain $\Omega$. For $r \in (0, \tfrac{1}{2}\rdo)$, let $\xi_r: \mathbb{R} \rightarrow \mathbb{R}$ be a smooth function such that
$$
\xi_r(t)= \left\{ 
\begin{aligned} 
&1 & &  \text{if } |t| \leq r,\\ 
&0 & & \text{if } 2r \leq |t|,
\end{aligned}
\right.
$$
$0 \leq \xi \leq 1$ and $|\xi_r^{\prime} | \leq c_2 / r$, where $c_2>1$ is a suitable constant.

Set $\widetilde \xi_r : \Omega \to \mathbb{R}$, defined by $\widetilde \xi_r(x) = \xi_r \circ \piu(x)$. Then,
\begin{equation}\label{118}
\bigl|\deriv \widetilde \xi_r(x) \bigr|_{\fnfs} = \bigl|\grad \widetilde \xi_r(x) \bigr|_{\fnf}\leq  \frac{c_2}{r}.
\end{equation}

Let $\zeta \in W^{1, p(x)}_{\loc}(\overline \Omega \setminus \Sigma ; \wghtv) \cap L^{\infty} _{\loc}(\overline \Omega \setminus \Sigma)$,    with  $\operatorname{supp}\zeta \subset \overline{\Omega} \setminus  \Sigma$. Then $(1-\widetilde \xi_r(x)) \zeta \in W_{\operatorname{loc}}^{1, p(x)}(\overline{\Omega} \setminus ( \Gamma \cup \Sigma ); \wghtv) \cap L_{\operatorname{loc}}^{\infty}(\overline{\Omega} \setminus (\Gamma \cup \Sigma))$, and $\operatorname{supp} (1-\widetilde{\xi} _r )\zeta \subset \overline{\Omega} \setminus (\Gamma \cup \Sigma)$. Therefore, \eqref{101} yields
\begin{equation} \label{76}
\begin{aligned}
 \int_{\Omega}   \mathscr{A}_p(\cdot , \grad u)\bullet\left[ ( 1-\widetilde \xi_r ) \deriv \zeta-\zeta \deriv \widetilde \xi _r  \right] \, \din \wghtv
 + \int_{\Omega} & \displaystyle f ( \cdot , u) (1-\widetilde \xi_r ) \zeta  \, \din \wghtv  \\[5pt] 
& \, \displaystyle +\int_{\partial \Omega  } b(\cdot , u)  \left(1- \widetilde \xi_r\right) \zeta \, \ds=0 .
\end{aligned}
\end{equation}
As $r \rightarrow 0^{+}$,   equality \eqref{76} implies \eqref{101}. Indeed, we have
$$
\begin{aligned}
 \lim _{r \rightarrow 0} \int_{\Omega  }    \mathscr{A}_p( \cdot , \grad u) \bullet \left[ \left(1-\widetilde \xi_r\right) \deriv \zeta\right]  + f (\cdot , u)\left(1-\widetilde \xi_r\right) \zeta     \, \displaystyle \din \wghtv + \int_{\partial \Omega} b(\cdot ,  u) \left(1-\widetilde \xi_r\right) \zeta & \, \ds \\[5pt]
 =  \int_{\Omega}      \mathscr{A}_p( \cdot , \grad u)\bullet \deriv \zeta  + f (\cdot , u) \zeta   \, \din \wghtv  + & \displaystyle \int_{\partial \Omega} b(\cdot , u) \zeta \, \ds .
\end{aligned}
$$
Additionally, by \ref{fnf21}, \ref{104}, \eqref{118}, and \eqref{28},
$$
\begin{aligned}
& \left|\int _{\Omega} \mathscr{A}_p(\cdot , \grad u) \bullet (\zeta   \deriv \widetilde \xi _r ) \, \din \wghtv  \right| \\[3pt]
& \leq  \displaystyle \frac{\Lambda _2 c _2}{r} \int _{\Omega   \cap\{r \leq \piu  \leq 2 r\}} |\grad u|^{p-1}_{\fnf}  \left( 1+  \pid  \right) ^{\tfrac{\boldsymbol{\delta} (p-1)}{2}}   |\zeta| \left( 1+  \pid  \right) ^{-\tfrac{\boldsymbol{\delta} (p-1)}{2}}  \, \din \wghtv \\[3pt]
&  \leq   \displaystyle \frac{c _3}{r}  \left\||\grad u|^{p-1}_{\fnf}\left( 1+  \pid  \right) ^{\tfrac{\boldsymbol{\delta} (p-1)}{2}} \right\| _{L^{\frac{p(x)}{p(x)-1}} (\Omega   \cap\{r \leq \piu  \leq 2 r\} ; \wghtv )} 
  \left\| \left( 1+  \pid  \right) ^{-\tfrac{\boldsymbol{\delta} (p-1)}{2}}\right\|_{L^{p(x)} (\Omega  \cap\{r \leq \piu  \leq 2 r\} ; \wghtv )} \\[3pt]
& \leq  \displaystyle \frac{c _3(2^{n_d} \ctb _1)^{1/{p^-}}}{r} \max \left\{ \left(\int _{\Omega   \cap\{r \leq \piu   \leq 2 r\}} |\grad u|^{p} _{\fnf}  \left( 1+  \pid  \right) ^{\tfrac{\boldsymbol{\delta} p^+}{2}}  \, \din \wghtv \right)^{\frac{p^+ -1}{p^+}}, \right. \\[3pt]
   & \displaystyle  \left. \left(\int _{\Omega   \cap\{r \leq \piu   \leq 2 r\}} |\grad u|^{p} _{\fnf}  \left( 1+  \pid  \right) ^{\tfrac{\boldsymbol{\delta} p^+}{2}} \, \din \wghtv \right)^{\frac{p^- -1}{p^-}} \right\} r^{\frac{n_d}{p^+}} \\[3pt]
& \leq  \displaystyle  c _4 r^{\frac{ (p^- -1)(n_d-p^{+}) - p^- }{p^-} + \frac{n_d}{p^+} } \\[5pt]
& = c_4 r^{n_d\left(1-\frac{1}{p^-}+\frac{1}{p^+}\right)-p^+ \left(1-\frac{1}{p^-}+\frac{1}{p^+}\right)} \rightarrow 0 \quad \text{as} \quad r \rightarrow 0,
\end{aligned}
$$
where
$$
c _3 = \Lambda _2 c _2 \|\zeta\|_{L^\infty (\Omega)}, \qquad c _4 = \left[c_1 \ctb _1 \left(\frac{5}{2} \right)^{n_d} \right]^{\frac{p^+ -1}{p^+}} c_3 (2^{n_d}\ctb _1)^{1/p^-}.
$$
So, we have obtained that equality \eqref{101} is fulfilled for all $\zeta \in W^{1, p(x)} _{\loc}(\overline \Omega \setminus \Sigma; \wghtv) \cap L^{\infty}_{\loc}(\overline \Omega \setminus \Sigma)$. Therefore, the singular set $\Gamma$ is removable for solutions of \eqref{1}.




\section{The limit problem as $p^+ \to 1$}

The proof of the Proposition \ref{302} follow from Lemma \ref{283} and Lemma \ref{303}. We begin with the following auxiliar result.
\begin{lemma} \label{250}
Suppose that $q_0\in (0 , qn)$. Let $u_p$ be a weak solution of \eqref{247}. Then there exist $n_0\in \mathbb{N}$, points $\{x_i\}_{i=1}^{n_0}\subset \overline{\mathscr{U}}$, and radii $\mathcal{R}_i\in(0,2\rdo/5)$, $i=1,\ldots,n_0$, such that
\begin{enumerate}[label=$(\roman*)$]
\item \label{252} 
$\overline{\mathscr{U}}\subset \bigcup_{i=1}^{n_0} B_{\mathcal{R}_i}(x_i)$.

\item \label{253} 
For each $i=1,\dots,n_0$ and for a.e. $x\in \Omega$ satisfying
$0<d(x_i,x)<\frac{5\mathcal{R}_i}{4}$, we have
\begin{equation}\label{258}
|u_p(x)|\le {\bf C}_{p^+}\, d(x_i,x)^{-\tau_{p^+}}.
\end{equation}
Moreover,
${\bf C}_{p^+}={\bf C}_{p^+}(\Lambda_1,\Lambda_2,n,q,\consgaglia,\gaglia,p^-, {\bf c}_1 ) >1$ and
$\tau_{p^+}=\tau_{p^+}(\Lambda_1,\Lambda_2,n,q,\gaglia,p^-,{\bf c}_1)>0$,
where ${\bf c}_1$ is given in Lemma \ref{64}. In addition,
\begin{equation}\label{289}
\lim_{p^+\to 1}{\bf C}_{p^+}>0,
\qquad
0<\lim_{p^+\to 1} \tau_{p^+}<n, 
\qquad
0<\lim_{p^+\to 1} q_0\tau_{p^+}<n .
\end{equation}

\item \label{254} 
There exists a constant $c=c(n,\fnf)>0$ such that, for every $\epsilon\in(0,n)$, $i\in\{1,\ldots,n_0\}$, and $r\in(0,5\mathcal{R}_i/4]$,
\begin{equation}\label{251}
\int_{B_r(x_i)} d(x_i,x)^{-\epsilon} \, \din \wghtv
\leq \frac{c}{n-\epsilon} r^{n-\epsilon}.
\end{equation}
\end{enumerate}
\end{lemma}

\begin{proof}
The proof follows by applying Corollary \ref{223} locally with
\[
{\piu}_y(x):=d(y,x),\qquad {\pid}_y\equiv 0,\qquad n_d=n,
\]
and by choosing the parameter \(\epsilon>0\) sufficiently small. In this way,
\ref{252}--\ref{254} follow from \eqref{256}, together with the local equivalence
between the Finsler distance and the Euclidean distance in coordinate charts:
for every point \(x\in\rmf\), there exist an open neighborhood \(U_x\), a
constant \(C\ge 1\), and a diffeomorphism
\[
\varphi:B_1(0)\subset\mathbb{R}^n\to U_x
\]
such that
\[
C^{-1}\lvert \mathbf{x}-\mathbf{y}\rvert
\le
d\bigl(\varphi(\mathbf{x}),\varphi(\mathbf{y})\bigr)
\le
C\lvert \mathbf{x}-\mathbf{y}\rvert,
\qquad \mathbf{x},\mathbf{y}\in B_1(0);
\]
see \cite[Chapter 1, p.~13]{zbMATH01624431}. \end{proof}


\begin{lemma}\label{283}
Let $u_p$ be the solution to problem \eqref{247}. Then there exist a subsequence $\{u_{p_m}\}_{m=1}^{\infty}$, ${\bf z} \in L^{\infty}(T\mathscr{U})$, and $u\in BV(\mathscr{U};\wghtv)$ such that
\begin{gather}
|\grad u_{p_m}|_{\fnf}^{p_m-2}\grad u_{p_m} \rightharpoonup \mathbf{z}
\quad \text{weakly in } L^s(T\mathscr{U};\wghtv) \text{ for every } 1 \leq s<\infty,\label{261} \\
u_{p_m} \to u \quad \text{strongly in } L^{1}(\mathscr{U};\wghtv), \label{273}\\
u_{p_m} \to u \quad \text{a.e. in } \mathscr{U},\label{279}\\
u\in L^{q+1}(\mathscr{U};\wghtv), \label{290} \\
\sup_m \left\| |\grad u_{p_m}|_{\fnf}^{p_m-1} \right\|_{L^\infty(\mathscr{U})} < \infty. \label{280}
\end{gather}

Moreover,
\begin{gather}
\||{\bf z}|_{\fnf}\|_{L^\infty(\mathscr{U})} \leq 1, \label{281}\\
-\operatorname{div} \mathbf{z}+|u|^{q-1}u= 0
\quad \text{in } \mathcal{D}^{\prime}(\mathscr{U}). \label{282}
\end{gather}
\end{lemma}

\begin{proof}
{\bf Step 1.} By Lemma \ref{250}, for $p^+$ close to $1$, we have
\begin{equation}\label{249}
\int _{\mathscr{U}} |u_p|^p\, \din \wghtv \leq {\bf C} _{p^+} ^{p^+}\sum _{i=1} ^{n_0} \int _{B_{\mathcal{R} _i} (x_i)} d(x_i ,x)^{-p^+\tau _{p^+}}\, \din \wghtv \leq  \frac{{\bf C} _{p^+} ^{p^+} c n_0}{n-p^+ \tau _{p^+}} \left( \max _i\mathcal{R} _i \right)^{n-\delta_1},
\end{equation}
where $\delta _1 := (n+ \lim _{p^{+} \rightarrow 1} p^{+} \tau_{p^{+}})/2 <n$ because of \eqref{289}.


Furthermore, if $r\in (0,\mathcal{R} _i/2]$, then by \eqref{257} and Lemma \ref{250},
\begin{align*}
\int_{\{r \leq d(x_i , \cdot) \leq 2 r\}}|\grad u _p|_{\fnf}^p\,  \din \wghtv \leq &\, c_0 r^{-p^{+}} \int_{\left\{r / 2 \leq d(x_i , \cdot) \leq  5 r / 2\right\}}|u_p|^p\, \din \wghtv\\[3pt]
\leq & \,  \frac{{\bf C}_{p^{+}}^{p^{+}} c c_0 }{n-p^{+} \tau_{p^{+}}} \left(\frac{5}{2} \right)^{n-\delta _1 } r^{n-\delta _1 -p^+},
\end{align*}
where $c_0$ is given in \eqref{257}.

Then, similarly to \eqref{246},
\begin{equation}\label{259}
\begin{aligned}
\int _{\{d(x_i , \cdot) \leq \mathcal{R} _i \}} |\grad u_p|^p_{\fnf} \, \din \wghtv \leq   & \,  c_0 \left(\frac{\mathcal{R} _i}{2}\right)^{-p^{+}} \sum _{j=0} ^\infty  \int _{\{2^{-j-1}\mathcal{R} _i \leq d(x_i , \cdot) \leq 2^{-j}\mathcal{R} _i\}} |u_p|^p \, \din \wghtv \\[3pt]
\leq   & \,  \frac{{\bf C}_{p^{+}}^{p^{+}} c c_0 }{n-p^{+} \tau_{p^{+}}} \left(\frac{5}{2} \right)^{n-\delta _1 } \sum _{j=0}^{\infty}\left(\frac{\mathcal{R} _i}{2^{j+1}}\right)^{n-\delta _1  -p^+}.
\end{aligned}
\end{equation}

Hence, by \eqref{249} and \eqref{259},
\begin{equation}\label{260}
 \int_{\mathscr{U} } |\grad u_p|^p _{\fnf} + |u_p|^p \,  \din \wghtv \leq c_{1,p^+} (\Lambda _1 , \Lambda _2 , n , n_0 , \fnf , \delta _1 , \max _{i} \mathcal{R} _i, \sup _{p^{-} \in (1,\min \{2,q+1\})} \mathtt{C}_{p^{-}} ),
\end{equation}
where $c_{1,p^+} \geq 1$ and the limit
$$
\lim _{p^+\to 1} c_{1,p^+}
$$
exists.

Similarly to \eqref{249}, applying Lemma \ref{250} with $q_0 = q+1$ ($q_0<qn$ by hypothesis \eqref{291}), we obtain
\begin{equation}
\int _{\mathscr{U}} |u_p|^{q+1} \, \din \wghtv \leq c_{2,p^+} (\Lambda_1, \Lambda_2, n, n_0, \fnf , \delta_2, \max _{i} \mathcal{R} _i  , \sup _{p^{-}  \in (1,  \min \{2,q+1\})} \mathtt{C}_{p^{-}}),
\end{equation}
where $\delta_2:=[n+\lim _{p^{+} \to 1} (q+1)\tau_{p^{+}}] / 2 <n$ by \eqref{289}, $c_{2, p^{+}} \geq 1$, and the limit
$$
\lim _{p^{+} \to 1} c_{2, p^{+}}
$$
exists.


\medskip

{\bf Step 2.} Let us now fix $s \in(1, \infty)$ and consider $1<p^-\leq p^+  <\frac{s}{s-1}$. By Hölder's inequality and \eqref{152},
\begin{align}
& \left[ \int _{\mathscr{U}}|\grad u_p|_{\fnf}^{(p-1) s} \,  \din \wghtv \right]^{\frac{1}{s}} \nonumber\\[3pt]
& \leq \left\|| \grad u_p|_{\fnf}^{(p-1)s}\right\|_{L^{\frac{p}{(p-1)s}}(\mathscr{U} ; \wghtv )} ^{\frac{1}{s}} \|1\|^{\frac{1}{s}}_{L^{\frac{p}{p-(p-1)s}}(\mathscr{U} ; \wghtv )} \nonumber \\[5pt]
& \leq \left[c_{1,p^+} ^{\frac{(p^+-1) s}{p^+ }} \right]^{\frac{1}{s}} \max\left\{\wghtv(\mathscr{U})^{\frac{1}{s}-1+\frac{1}{p^{-}}} , \wghtv(\mathscr{U})^{\frac{1}{s}-1+\frac{1}{p^{+}}} \right\}. \label{263}
\end{align}

From this, we infer that these families are bounded. Thus, for $s>1$, there exists a subsequence $\{u_{p_{s,m}}\}$ and there exists $\mathbf{z}_s \in L^s ( T\mathscr{U} ; \wghtv )$ such that
$$
|\grad u_{p_{s,m}}|_{\fnf} ^{p_{s,m}-2} \grad u_{p_{s,m}}  \rightharpoonup \mathbf{z}_s \quad \text { weakly in } L^s( T\mathscr{U}  ; \wghtv).
$$

Since these facts hold for every $s$, two diagonal arguments allow us to find ${\bf z} \in L^s( T\mathscr{U} ; \wghtv)$  for all $s \in(1, \infty)$, together with a subsequence $\{u_{p_m}\}$ satisfying \eqref{261}. From \eqref{260} and Proposition \ref{227}, we may assume that $\{u_{p_m}\}$ satisfies \eqref{273} and \eqref{279}. Also, from \eqref{263} we obtain \eqref{280}.

Moreover, bearing in mind the lower semicontinuity of the $s$-norm with respect to weak convergence, we may let $p_m ^+$ go to $1$ in \eqref{263}; this yields
$$
\left(\int_{\mathscr{U}}|{\bf z}|_{\fnf}^s \, \din \wghtv\right)^{\frac{1}{s}}  \leq \wghtv (\mathscr{U})^{\frac{1}{s}},
$$
for every $s \in(1, \infty)$. We deduce that
$$
\||{\bf z}|_{\fnf} \|_{L^\infty (\mathscr{U})}  \leq 1.
$$

The validity of \eqref{282} follows simply by taking $\zeta \in C_c^{\infty}(\mathscr{U})$ as a test function in \eqref{248} (with $p=p_m$) and letting $m \to \infty$.\end{proof}

\begin{lemma} \label{303} Let $u_{p_m}$ be a solution of \eqref{247} and let ${\bf z} \in L^{\infty} (  T\mathscr{U})$ and $u\in BV(\mathscr{U} ; \wghtv)$ as in  Lemma \ref{283}. Then it holds
$$
(\mathbf{z}, D u)=\|D u\|_{v} \quad \text { as measures in } \mathscr{U} \text {. }
$$
\end{lemma}
\begin{proof}  {\bf Step 1.} For $k>0$, we define the truncation function $T_k : \mathbb{R} \to \mathbb{R}$ by
$$
T_k(s):= \max \big\{-k , \min \{s,k\}\big\}.
$$

Let us take $T_k(u_{p_m}) \zeta$, with $0 \leq \zeta \in C_c^1(\mathscr{U})$, as a test function in \eqref{248}. This gives
\begin{equation*}
\begin{aligned}
\int_{\mathscr{U}}\zeta |\grad T_k(u_{p_m})|_{\fnf} ^{{p_m}}  \, \din \wghtv
+\int_{\mathscr{U}} T_k(u_{p_m}) |\grad u_{p_m}|_{\fnf} ^{p_m-2} \grad u_{p_m} \bullet \deriv \zeta \, \din \wghtv &\\[3pt]
+\int_{\mathscr{U}} |u_{p_m}|^{q-1}u_{p_m} T_k(u_{p_m}) \zeta \, \din \wghtv &= 0.
\end{aligned}
\end{equation*}

By Young's inequality, 
$ab \leq \frac{1}{p_m}|a|^{p_m}+\frac{p_m-1}{p_m}|b|^{\frac{p_m}{p_m-1}}$, 
we obtain
\begin{equation} \label{284}
\begin{aligned}
\int_{\mathscr{U}} |\grad T_k(u_{p_m})|_{\fnf} \zeta \, \din \wghtv
+\int_{\mathscr{U}} T_k(u_{p_m}) |\grad u_{p_m}|_{\fnf} ^{p_m-2} \grad u_{p_m} \bullet \deriv \zeta \, \din \wghtv &\\[3pt]
+\int_{\mathscr{U}} |u_{p_m}|^{q-1}u_{p_m} T_k(u_{p_m}) \zeta \, \din \wghtv
&\leq \int_{ \mathscr{U} } \frac{p_m-1}{p_m}\,\zeta \, \din \wghtv.
\end{aligned}
\end{equation}

From \cite[Chapter 4, Proposition 15.1]{zbMATH06587358} (see also \cite[Layer-cake representation theorem]{zbMATH01601796}), we have
\begin{equation}\label{285}
\begin{aligned}
\int_{\mathscr{U}} \zeta |\grad T_k(u_{p_m})|_{\fnf}  \, \din \wghtv
= & \, \int_0^\infty \int_{\{\zeta >t \}} |\grad T_k(u_{p_m})|_{\fnf} \, \din \wghtv \, \din t\\[3pt]
\geq & \, \int_0^\infty \|D T_k(u_{p_m})\|_{v} (\{\zeta >t\})\, \din t .
\end{aligned}
\end{equation}

By Proposition \ref{226}, together with \eqref{285} and Lemma \ref{283}, taking $m \to \infty$ in \eqref{284}, we obtain
$$
\int_{\mathscr{U}} \zeta \, \din \|DT_k u\|_{v}
+ \int_{\mathscr{U}} T_k(u) {\bf z} \bullet \deriv \zeta \, \din \wghtv
+ \int_{\mathscr{U}} |u |^{q-1}u  T_k(u) \zeta \, \din \wghtv \leq 0 .
$$

Hence, letting $k \to \infty$,
$$
\int_{\mathscr{U}} \zeta \, \din \|D u\|_{v}
+ \int_{\mathscr{U}} u {\bf z} \bullet \deriv \zeta \, \din \wghtv
+ \int_{\mathscr{U}} |u|^{q+1} \zeta \, \din \wghtv \leq 0 .
$$

From \eqref{282}, $\operatorname{div} {\bf z}=|u|^{q-1}u$. Then,
\begin{equation}\label{294}
\int_{\mathscr{U}} \zeta \, \din \|D u\|_{v}
\leq -\int_{\mathscr{U}} u \mathbf{z} \bullet \deriv \zeta \, \din \wghtv
- \int_{\mathscr{U}} u \zeta \operatorname{div} {\bf z} \, \din \wghtv
= \int_{ \mathscr{U} }(\mathbf{z}, D u) \zeta \, \din \wghtv.
\end{equation}

\medskip

{\bf Step 2.} Let $\zeta \in C_c^1(\mathscr{U} )$ with $\zeta \geq 0$. By \eqref{290} in Lemma \ref{283}, we have $u\in L^{q+1}(\mathscr{U} ; \wghtv)$. Consequently, by Lemma \ref{286}, there exists a sequence
\[
u_n\in \operatorname{Lip} (\mathscr{U})\cap BV (\mathscr{U} ; \wghtv)
\]
such that
\begin{equation}\label{295}
\begin{gathered}
u_n \to u \quad \text{in } L^{1}(\mathscr{U} ; \wghtv ),\\
u_n \to u \quad \text{in } L^{q+1}(\mathscr{U} ; \wghtv ).
\end{gathered}
\end{equation}

Also, from \eqref{282} in Lemma \ref{283}, we see that
\begin{equation}\label{296}
\operatorname{div} {\bf z} \in L^{\frac{q+1}{q}} (\mathscr{U} ; \wghtv ).
\end{equation}

We have
$$
\int_{\mathscr{U}}(\mathbf{z}, D u_n) \zeta \, \din \wghtv
= \int_{\mathscr{U}} \zeta \mathbf{z}\bullet \deriv u_n  \, \din \wghtv
\leq \int_{\mathscr{U}} \zeta |{\bf z}|_{\fnf } | \deriv u_n |_{\fnfs}  \, \din \wghtv
\leq \int_{\mathscr{U}} \zeta  | \grad u_n |_{\fnf}  \, \din \wghtv,
$$
because $\||{\bf z}|_{\fnf}\|_{L^\infty (\mathscr{U})} \leq 1$.

From \eqref{295}, \eqref{296}, and \eqref{292}, by passing to the limit, we obtain
\begin{equation}\label{293}
\int_{\mathscr{U}}(\mathbf{z}, D u) \zeta \, \din \wghtv\leq \int_{\mathscr{U}} \zeta \, \din  \| D u \|_{v} .
\end{equation}

Therefore, from \eqref{294} and \eqref{293}, we conclude the proof of the lemma. \end{proof}



\vspace{1.5cm}


\noindent {\bf Funding:} This work was supported by CNPq - Conselho Nacional de Desenvolvimento Científico e Tecnológico,  Grant 153232/2024-2.






\bibliographystyle{abbrv}
    
\bibliography{ref}

\end{document}